\documentclass[final,leqno,onefignum,onetabnum,dvipsnames]{siamltex1213}
\usepackage[ruled,vlined, linesnumbered]{algorithm2e}
\usepackage{amsmath,url,cases,supertabular,longtable}
\usepackage{amssymb,mathtools,multirow}
\usepackage{enumerate,makecell}
\usepackage{amsfonts}
\usepackage{graphicx}
\usepackage{verbatim}
\usepackage{mathrsfs,subeqnarray,subfigure}
\usepackage{listings}
\usepackage{subfigure}
\usepackage{xcolor}

\newcommand{\bit}{\begin{itemize}}
\newcommand{\eit}{\end{itemize}}
\newcommand{\be}{\begin{equation}}
\newcommand{\ee}{\end{equation}}
\newcommand{\bea}{\begin{eqnarray}}
\newcommand{\eea}{\end{eqnarray}}

\newcommand{\Rbb}{\mathbb{R}}
\newcommand{\st}{\mathrm{s.t.}}
\newcommand{\Dcaln}{\mathcal{D}_n}
\newcommand{\Ncal}{\mathcal{N}}
\newcommand{\Lcal}{\mathcal{L}}

\newcommand{\argmin}{\mathop{\mathrm{arg\,min}}}

\newcommand{\tr}{\mathrm{tr}}

\newcommand{\nn}{\nonumber}
\newcommand{\proj}{\mathcal{P}}
\newcommand{\tX}{\tilde{X}}

\newcommand{\tol}{\mathrm{tol}}
\newcommand{\best}{\mathrm{best}}

\newcommand{\Rmn}[1]{\uppercase\expandafter{\romannumeral#1}}
\renewcommand{\theenumi}{\roman{enumi}}

\newcommand{\Fsf}{\mathsf{F}}
\newcommand{\Tsf}{\mathsf{T}}

\newcommand{\1}{\mathbf{1}}
\newcommand{\0}{\mathbf{0}}
\newcommand{\uf}{\underline{f}}

\newcommand{\e}{\mathbf{e}}

\newcommand{\Acal}{\mathcal{A}}

\newcommand{\bX}{\bar{X}}

\newcommand{\bW}{\bar{W}}
\newcommand{\Diag}{\mathop{\mathrm{Diag}}}

\newcommand{\Dcalonen}{\mathcal{D}_{[0,1]^n}}

\newcommand{\um}{\underline{m}}
\newcommand{\om}{\overline{m}}

\newcommand{\revise}[1]{{\color{black}#1}}

\makeatletter
\def\LT@makecaption#1#2#3{%
 \LT@mcol\LT@cols c{\hbox to\z@{\hss\parbox[c]\LTcapwidth{%
     \footnotesize
     \setlength{\parindent}{1.5pc}
   \hbox to\hsize{\hfil #1{\normalfont \scshape #2 }  \hfil}
    \setbox\@tempboxa\hbox{{\normalfont\itshape #3}}
   \ifdim\wd\@tempboxa>\hsize
    {\hspace{0pt}\normalfont\itshape #3}\par 
   \else
     \hbox to\hsize{\hfil\box\@tempboxa\hfil}%
   \fi
    \vskip\belowcaptionskip
  }%
 \hss
 } } }

\makeatother

\title{ $L_\MakeLowercase{p}$-norm regularization algorithms for optimization \\over permutation matrices
}
\author{Bo Jiang\thanks{School of Mathematical Sciences, Key Laboratory for NSLSCS of Jiangsu Province,
Nanjing Normal University, Nanjing 210023, CHINA (jiangbo@njnu.edu.cn). Research supported in part by
the NSFC grant 11501298,
{the NSF of Jiangsu Province (BK20150965),  and  the Priority Academic Program Development of Jiangsu Higher Education Institutions.}
}\and Ya-Feng Liu\thanks{State Key Laboratory of Scientific and Engineering Computing, Institute of Computational
Mathematics and Scientific/Engineering Computing, Academy of Mathematics and Systems Science, Chinese Academy of Sciences,
Beijing 100190, CHINA (yafliu@lsec.cc.ac.cn). Research supported in part by NSFC grants 11301516, 11331012, and 11571221.}
\and Zaiwen Wen\thanks{Beijing International Center for Mathematical Research, Peking University,  Beijing 100871, CHINA (wenzw@pku.edu.cn). Research supported in part by NSFC grants 11322109 and 11421101, and by the National
Basic Research Project under the grant 2015CB856002.}}
\date{\today}
\begin{document}
\maketitle
\slugger{mms}{xxxx}{xx}{x}{x--x}

\begin{abstract}
Optimization problems over permutation matrices appear widely in facility
layout, chip design, scheduling, pattern recognition, computer vision, graph matching, etc.  Since this problem is
 NP-hard due to the combinatorial nature of permutation matrices,
we relax the variable to be the more tractable doubly stochastic matrices and add an $L_p$-norm ($
0 < p < 1$) regularization term to the objective function.  The optimal
solutions of the $L_p$-regularized problem are the same as the original problem
if the regularization parameter is sufficiently large. A lower bound estimation
of the nonzero entries of the stationary points and some connections between the
local minimizers and the permutation matrices are further established.  Then we
propose an $L_p$ regularization algorithm with local refinements. The algorithm approximately solves a sequence of $L_p$ regularization subproblems by the projected gradient
method using a nonmontone line search with the Barzilai-Borwein step sizes. Its performance can be further improved if it is combined with certain local search methods, the cutting plane techniques as well as a new negative proximal point scheme. Extensive numerical results on QAPLIB  and the bandwidth minimization problem  show that our proposed algorithms can often find reasonably high quality solutions within a competitive amount of time.
\end{abstract}


\pagestyle{myheadings}
\thispagestyle{plain}
\markboth{\sc{optimization over permutation matrices}}{\sc{Bo Jiang, Ya-Feng Liu, and Zaiwen Wen}}

 \section{Introduction}
In this paper, we consider optimization  over permutation matrices:
 \be \label{equ:prob:optperm}
    \min\limits_{X \in \Pi_n} \, f(X),
 \ee
 where $f(X) \colon \Rbb^{n \times n} \rightarrow \Rbb^n$ is  differentiable and $\Pi_n$ is the set of $n$-order permutation matrices, namely,
 \[\Pi_n =  \{X \in \Rbb^{n \times n} \mid X\e = X^{\Tsf} \e = \e,  X_{ij} \in \{0,1\}\},\]
in which  $\e\in \Rbb^n$ is a vector of all ones.  Given  two groups of
correlative objects associated with the rows and columns of a square matrix,
respectively, each permutation matrix implies an assignment from objects in one
group to objects in the other group. Thus problem \eqref{equ:prob:optperm} is
also referred to as the nonlinear assignment problem since it looks for the best
assignment with the smallest nonlinear cost among the two groups.

One of the most famous special cases of problem \eqref{equ:prob:optperm} is the  quadratic assignment problem (QAP) \cite{koopmans1957assignment}, one of the hardest combinatorial optimization problems: \begin{equation} \label{equ:prob:qap}
   \min\limits_{X \in \Pi_n} \, \tr(A^{\Tsf} X B X^{\Tsf}),
 \end{equation}
 where $A, B \in \Rbb^{n \times n}$. The QAP has wide applications in
 statistics, facility layout, chip design, keyboards design, scheduling, manufacturing, etc. For more details, one can refer to \cite{burkard2013quadratic, drezner2015quadratic} and references therein.  Some generalizations of QAP are also investigated, such as the cubic, quartic, and generally $N$-adic assignment problems, see \cite{lawler1963quadratic}.  Burkard {\it et al.} \cite{qela1994biquadratic} discussed one application of biquadratic assignment, a special quartic assignment,  in very-large-scale integrated circuit design.
 For other applications of general nonlinear assignment, one can refer to \cite{samworth2012independent}.

{It is generally not practical to solve problem \eqref{equ:prob:optperm} exactly due to its strong NP-hardness.  The goal of this paper is to develop fast algorithms which can find high quality permutation matrices for problem \eqref{equ:prob:optperm}, especially when the problem dimension is large.}

\subsection{Related works}
Although most of the existing methods are mainly tailored for QAP
\eqref{equ:prob:qap}, many of them can be easily extended to the general problem
\eqref{equ:prob:optperm}.  Thus in this subsection, we simply review some
related works on QAP.  Since in this paper we are interested in methods which
can quickly find high quality solutions, we shall not introduce the lower bound
methods  \cite{kim2015lagrangian, xia2006new} and the exact methods \cite{adams2007level, fischetti2012three} in details. For these algorithms, we refer the readers to the review papers \cite{anstreicher2003recent, burkard2013quadratic, burkard1999quadratic,  drezner2015quadratic, flnke2011quadratic,  loiola2007survey}. We next briefly introduce a few approximation methods, which often return good approximate solutions (which
are permutation matrices) for QAP \eqref{equ:prob:qap} in a reasonable time.

The philosophy of the vertex based methods is similar to that
of the simplex methods for linear programming.  Specifically, it updates
the iterates from one permutation matrix to another.  Different strategies of
updating permutation matrices lead to different methods, including local search
methods \cite{frieze1989algorithms, murthy1992local}, greedy randomized adaptive
search procedures \cite{feo1995greedy, pardalos1994greedy}, tabu search methods  \cite{battiti1994reactive, taillard1991robust}, genetic algorithms \cite{ahuja2000greedy, tate1995genetic}.  For a comprehensive review on the vertex based methods, one can refer to \cite{burkard2013quadratic} and references therein.

The interior-point based methods generates a permutation matrix along some ``central'' path (of interior points) of the set of doubly stochastic matrices, also known as the Birkhoff polytope:
\[\Dcaln =  \{X \in \Rbb^{n \times n} \mid X\e = X^{\Tsf} \e = \e,  X \geq \0\},\]
where the symbol $\geq$ denotes the componentwise ordering and  $\0 \in \Rbb^{n \times n}$ is the all-zero matrix.  By the Birkhoff-von Neumann theorem \cite{birkhoff1946tres}, we know that $\Dcaln$ is the convex hull of $\Pi_n$ and the set of vertices of $\Dcaln$ is exactly $\Pi_n$.
 Similar to the interior-point methods for linear programming, these methods construct a path of interior points by solving a
sequence of regularization problems over $\Dcaln$.

Xia \cite{xia2010efficient} proposed a Lagrangian smoothing algorithm for QAP \eqref{equ:prob:qap} which solves a sequence of $L_2$  regularization subproblem
\be \label{equ:prob:qap:l2}
\min_{X \in \Dcaln} \, \tr(A^{\Tsf} X B X^{\Tsf}) + \mu_k \|X\|_{\Fsf}^2
\ee
with dynamically decreasing parameters $\left\{\mu_k \mid k = 0, 1, 2, \ldots \right\}.$ The
initial value $\mu_0$ in subproblem \eqref{equ:prob:qap:l2} was chosen such that the problem is convex and the subproblem with fixed $\mu_k$ was approximately solved by the Frank-Wolfe algorithm up to fixed number of iterations (it was called the truncated Frank-Wolfe method in \cite{xia2010efficient}).
  Note that Bazaraa and Sherali  \cite{bazaraa1982use} first considered
  regularization problem \eqref{equ:prob:qap:l2} where $\mu_k$ is chosen such
  that the problem is strongly concave, but they only solved
  \eqref{equ:prob:qap:l2} once.  In addition, Huang \cite{huang2008continuous}
  used the quartic term $\|X \odot (\1 - X) \|_{\Fsf}^2$ to construct the regularization problem, where  $\odot$ is the
  Hadamard product  and $\1 \in \Rbb^{n \times n}$ is the matrix of all ones.

One special case of QAP \eqref{equ:prob:qap} (and thus problem \eqref{equ:prob:optperm}) is the graph matching problem:       \begin{equation} \label{equ:prob:gmp}
   \min\limits_{X \in \Pi_n} \, \|AX - XB\|_{\Fsf}^2,
 \end{equation}
 which has wide applications in pattern recognition, computer vision, etc
 \cite{conte2004thirty}.  Note that the objective function in
 \eqref{equ:prob:gmp} is always convex in $X$ and can be rewritten as $-2\tr(A^{\Tsf} X B X^{\Tsf}) + \|A\|_{\Fsf}^2 + \|B\|_{\Fsf}^2$, which is no longer convex in $X$.  By constructing a convex and a
 concave quadratic optimization problem over doubly stochastic matrices,
 \cite{liu2012extended, zaslavskiy2009path} considered  path following
 algorithms which solve a sequence of convex combinations of the convex and
 concave optimization problems, wherein the combination parameters were chosen
 such that the resulting problems change from convex to concave.  Moreover, \cite{lyzinski2014graph}  observed that initialization with a convex problem can improve the practical performance of the regularization algorithms.

Recently, Fogel {\it et al.} \cite{fogel2013convex} used QAP  to solve the seriation and 2-SUM problems. QAP considered therein takes the following special structure:
\begin{equation} \label{equ:prob:qap:vector}
   \min\limits_{X \in \Pi_n} \,  \pi_I X^{\Tsf} L_A X \pi_I,
 \end{equation}
where $A$ is binary symmetric, $\pi_I = (1, 2, \ldots, n)^{\Tsf}$, and  $L_A =
\Diag(A \e) - A$. Here, $\Diag(A\e)$ represents a diagonal matrix whose diagonal
elements are $A\e$. They  used the regularization term $\|P X\|_{\Fsf}^2$ with $P = I_n - \frac1n \1$, where $I_n$ denotes the identity matrix of size $n$, to construct  a convex relaxation problem over $\Dcaln$ for \eqref{equ:prob:qap:vector}.  Using a recent result of Goemans \cite{goemans2015}, Lim and Wright
\cite{lim2014beyond}  constructed a new convex relaxation problem over a convex
hull of the permutation vectors for \eqref{equ:prob:qap:vector}. In their
problem,  the number of variables and constraints  reduces to $O(n \log n)$ in
theory and $O(n \log^2 n)$ in practice.   The advantage of this new formulation
was illustrated by numerical experiments. However, as pointed in
\cite{lim2014beyond},  it is still unclear how to extend their  new formulation
to the more general QAP \eqref{equ:prob:qap}.

The aforementioned regularization or relaxation methods all considered solving optimization over $\Dcaln$. By using the fact that $\Pi_n  = \{X \mid X^{\Tsf} X = I_n, X \geq 0\}$, Wen and Yin \cite{wen2013feasible} reformulated QAP \eqref{equ:prob:qap} as
\be
  \min_{X \in \Rbb^{n \times n}}\, \tr(A^{\Tsf} (X \odot X) B (X \odot X)^{\Tsf}) \quad \st \quad X^{\Tsf} X = I_n,~X \geq 0  \nn
\ee
and proposed an efficient algorithm by applying the augmented Lagrangian method to handle the constraint $X \geq 0$.  In their algorithm, a sequence of augmented Lagrangian subproblems with orthogonality constraint is solved inexactly.

\subsection{Our contribution}  In this paper, we consider optimization problem
\eqref{equ:prob:optperm} over permutation matrices and propose an $L_{p}$-norm
regularization model with $0 < p< 1$.  The $L_{p}$-norm regularization is exact
in the sense that its optimal solutions are the same as the original problem if
the regularization parameter is sufficiently large.  We derive a lower bound
estimate for nonzero entries of the stationary points and establish connections
between the local minimizers and the permutation matrices.  Then we propose an
algorithm with local refinements.  Basically, it solves approximately a sequence
of $L_p$ regularization subproblems by  the projected gradient method using nonmontone line search
with the Barzilai-Borwein (BB) step sizes.  The projection of a matrix onto the
set of doubly stochastic matrices is  solved by a dual gradient method which is often
able to find a highly accurate solution faster than the commercial solver MOSEK.
The performance of our algorithm can be further improved if it is combined with
certain local search methods, cutting plane techniques as well as a new
negative proximal point scheme. In fact, we show that the global solution of the
regularized model  with a negative proximal point term is also the optimal
solution of problem \eqref{equ:prob:optperm}. In particular, it is the farthest
point in the solution set of  \eqref{equ:prob:optperm} to the given point in the
proximal term. Therefore, this technique enables us to get rid of a large
portion  of non-optimal permutation matrices.
 Numerical results on QAPLIB  show that the
proposed algorithm can quickly find high quality approximate solutions. For
example,  our proposed
algorithm can find a solution of the largest problem instance ``tai256c'' in
QAPLIB with a gap $0.2610\%$ to the best known solution in
 less than a half minute on an ordinary  PC.

\subsection{Notation}
For any $X \in \Rbb^{n \times n}$,  $\|X\|_0$ denotes the number of nonzero
entries of $X$ and its $L_{p}$-norm is $\|X\|_p^p = \sum_{i=1}^n \sum_{j = 1}^n |X_{ij}|^p$ for $0< p < 1$.
The matrix $X^p \in \Rbb^n$ is defined by $(X^p)_{ij} = X_{ij}^p$ for each $(i, j) \in \Ncal \times \Ncal$ with $\Ncal := \{1, \ldots, n\}.$
The inner product of $M, N\in \Rbb^{m \times n}$ is defined as $\langle M, N \rangle = \tr(M^{\Tsf} N),$ where $\tr(\cdot)$ is the trace operator. The Kronecker product of any two matrices $M$ and $N$ is denoted as $M \otimes N$.
We define $\Dcalonen = \{X \in \Rbb^{n \times n} \mid \0 \leq X \leq \1\}$, where the symbol $\leq$ denotes the  componentwise ordering.
  For any $X \in \Dcaln$, let the support set of $X$ be $\Lambda(X) = \{(i,j)\in \Ncal \times \Ncal \mid X_{ij} > 0\}$.
    Denote $\Lambda(X_{i \cdot}) = \{j \in \Ncal \mid (i,j)\in \Lambda(X)\}$ and $\Lambda(X_{\cdot j}) = \{i \in \Ncal \mid (i,j)\in \Lambda(X)\}$.
 For any $X \in \Pi_n$, we  define its corresponding permutation vector as
 $\pi \in \Rbb^n$ with the $j$-th  element being  $\pi_j = i$, where $i \in
 \Ncal$ satisfies $X_{ij} = 1$.  We denote the 2-neighborhood of $X$ as
 $\Ncal_2 (X) = \{Z \in \Pi_n: \|Z - X \|_{\Fsf}^2 \leq 4\} = \{Z \in \Pi_n \mid \langle Z,  X \rangle \geq n - 2\}$ and
 $\Ncal_2^{\best}(X)$ is the set of permutation matrices which take the least function value, in the  sense of $f(\cdot)$, among all permutation matrices in $\Ncal_2 (X)$.  We say that $X$ is locally 2-optimal if $X \in \Ncal_2^{\best}(X)$.

\subsection{Organization} The rest of this paper is organized as follows. Some
preliminaries are provided in Section \ref{section:preliminaries}. The $L_0$
regularization problem is introduced in Section \ref{subsection:l0} and the
extension to the general $L_p$  regularization problem is presented in Section
\ref{subsection:lp}. Our $L_p$ regularization algorithmic framework is given in
Section \ref{subsection:lp:alg:frame} and  a practical $L_p$
regularization algorithm is developed in Section
\ref{subsection:lp:alg:practical}.  A fast dual BB method for computing the
projection onto the set of doubly stochastic matrices is proposed in Section
\ref{subsection:dualBB:proj}. We combine the $L_p$ regularization algorithm with
the cutting plane technique in Section \ref{subsection:lp:cp} and propose a negative proximal point technique in Section \ref{subsection:lp:negProx}.
 Some implementation details, including an efficient way of doing local
 2-neighborhood search, are given  in Section \ref{subsection:impDetails}, while
 numerical results on QAPLIB are reported in Sections \ref{subsection:lp:comp:0}
 -- \ref{subsection:lp:comp}. We apply our proposed algorithms to solve
 a class of real-world problems, namely, the bandwidth minimization problem in Section \ref{section:bm}. Finally, we make some concluding remarks in Section \ref{secion:conclusion}.

\section{Preliminaries} \label{section:preliminaries}
Assume that $\nabla f(X)$ is Lipschitz continuous on $\Dcalonen$ with Lipschitz constant $L$, that is,
\be \label{equ:grad:Lips0}
\| \nabla f(X) - \nabla f(Y) \|_{\Fsf} \leq L \|X - Y\|_{\Fsf},  \quad \forall \ X, Y \in \Dcalonen.
\ee
Moreover, we assume
\be\label{equ:f:concave}
 \frac{\underline{\nu}_f}{2} \|Y - X\|_{\Fsf}^2 \le f(Y) - f(X) - \langle \nabla f(X), Y - X \rangle  \le \frac{\overline{\nu}_f}{2} \|Y - X\|_{\Fsf}^2, \quad \forall \ X, Y \in  \Dcaln.
\ee
 Note that when $\overline{\nu}_f < 0$, \eqref{equ:f:concave} implies that $f(\cdot)$ is strongly concave.
By some easy calculations, \eqref{equ:f:concave} indicates that
\be\label{equ:grad:Lips}
f(t X + (1-t)Y) \geq t f(X) + (1 -t)f(Y) - \frac{\overline{\nu}_f}{2} t(1 - t) \|X - Y\|_{\Fsf}^2
\ee
holds for all $X, Y \in \Dcaln$ and $t \in [0,1]$.  Notice that
assumptions \eqref{equ:grad:Lips0} and \eqref{equ:f:concave} are mild. It can be easily verified that the objective in QAP \eqref{equ:prob:qap} satisfies the two assumptions with $L = \max\{|\overline{\nu}_f|, |\underline{\nu}_f|\} $, where $\overline{\nu}_f$ and $\underline{\nu}_f$  are taken as the largest and smallest eigenvalues of $B^{\Tsf} \otimes A^{\Tsf} + B \otimes A \in \Rbb^{n^2 \times n^2}$, respectively.   

Consider the function  $h(x) = (x+\epsilon)^p$ over $x \in [0,1]$ with $0< p <
1$ and $\epsilon \geq 0$. It is  straightforward  to show that $h(\cdot)$ is
strongly concave with the parameter
  \be \label{equ:nu_h}
  \overline{\nu}_h \coloneqq p(1 - p)(1 + \epsilon)^{p-2}
  \ee
 on the interval [0,1].  Moreover, the inequality $h(x) - h(y)  > \langle h'(x), x - y \rangle$ holds for any $x \neq y \in [0,1]$ and $ x + \epsilon >0$.  Consider the function
\be \label{equ:hX}
h(X) \coloneqq  \|X + \epsilon \1\|_p^p = \sum_{i =1}^n \sum_{j=1}^n (X_{ij} + \epsilon)^p.
\ee
 Using the strong concavity of $h(x)$ and the separable structure of $h(X)$, we can  establish the following result.
\begin{proposition}\label{lemma:concave:lp}
  The function $h(X)$ with $0< p < 1$ and $\epsilon \geq 0$ is strongly concave with parameter $\overline{\nu}_h$ on $\Dcalonen$, namely,
  \be\label{equ:lemma:concave:lp}
h(tX + (1-t) Y) \geq th(X) + (1-t) h(Y) + \frac{\overline{\nu}_h}{2} t(1-t)\|Y - X\|_{\Fsf}^2
\ee
holds for any $X, Y \in \Dcalonen$.
\end{proposition}

\section{$L_p$-norm regularization model for \eqref{equ:prob:optperm}} \label{section:lp}
In this section, we first  introduce the $L_0$-norm regularization model for \eqref{equ:prob:optperm} in
Section \ref{subsection:l0}. Then, we focus on the  $L_p$-norm regularization model for \eqref{equ:prob:optperm} in Section \ref{subsection:lp}.

\subsection{$L_0$-norm regularization}\label{subsection:l0}
Observe that $\Pi_n$ can be equivalently characterized as
 \be \label{equ:Pi:l0}
 \Pi_n = \Dcaln \cap \{X \mid \|X\|_0 = n \}.
 \ee
 Thus problem \eqref{equ:prob:optperm} is equivalent to the problem of minimizing $f(X)$ over $\Dcaln$ intersecting with the sparse constraint $\|X\|_0 = n$.
 This fact motivates us to consider the $L_0$ regularization model for \eqref{equ:prob:optperm} as follows:
\be\label{equ:prob:qap:l0}
\min_{X \in \Dcaln} \, f(X) + \sigma \|X\|_0,
\ee
where $\sigma > 0$.
Let $X^*$ be  one global solution of  \eqref{equ:prob:optperm}. Denote
\be \label{equ:fstart:fbar}
f^* \coloneqq f(X^*) \quad \mbox{and}\quad
\uf \coloneqq \min\limits_{X \in \Dcaln} \, f(X).
\ee We immediately have $f^* \geq \uf$ and obtain the following lemma.

\begin{lemma}\label{lemma:l0:exact}
  Suppose $\sigma > \frac{1}{2}(f^* - \uf),$
  where $f^*$ and $\uf$ are given in \eqref{equ:fstart:fbar}. Then any global solution $X(\sigma)$ of \eqref{equ:prob:qap:l0}  is also the global solution of \eqref{equ:prob:optperm}.
\end{lemma}
\begin{proof}
First,  we use the contradiction argument to show that $X(\sigma)$ is a permutation matrix. Suppose that $X(\sigma)$ is not a permutation matrix, then  $\|X(\sigma)\|_0 \geq n + 2$. From the optimality of $X(\sigma)$, we know that
\be \label{equ:lemma:l0:exact:a1}
f^* + \sigma \|X^*\|_0 \geq f(X(\sigma)) + \sigma \|X(\sigma)\|_0.
\ee Moreover, by the definition of $\uf$, we have
$ f(X(\sigma)) + \sigma \|X(\sigma)\|_0 \geq \uf +  \sigma \|X(\sigma)\|_0$.
This inequality, together with \eqref{equ:lemma:l0:exact:a1} and the fact $\|X(\sigma)\|_0 \geq n + 2$, implies
\[
\sigma \leq (f^* - \uf)/\left(\|X(\sigma)\|_0 - \|X^*\|_0\right) \leq (f^* - \uf)/2,
\]
which is a contradiction to $\sigma > \frac12 (f^* - \uf)$.

Using $\|X(\sigma)\|_0 =
\|X^*\|_0 = n$ and \eqref{equ:lemma:l0:exact:a1}, we have $ f^* \geq f(X(\sigma))$, which means that $X(\sigma)$ is the global solution of \eqref{equ:prob:qap}.  \end{proof}

Lemma \ref{lemma:l0:exact} shows that $L_0$ regularization
\eqref{equ:prob:qap:l0} is exact in the sense that $L_0$ regularization problem
\eqref{equ:prob:qap:l0} shares the same global solution with problem
\eqref{equ:prob:optperm}. However, $\|X\|_0$ is not continuous in $X,$ which
makes the $L_0$-norm regularization problem \eqref{equ:prob:qap:l0} hard to be
solved.

\subsection{$L_p$-norm regularization}\label{subsection:lp}
\ The $L_1$-norm regularization $\|X\|_1$ has been shown in
\cite{candes2006stable, donoho2006compressed} to be a good approximation of
$\|X\|_0$ under some mild conditions. However, for any $X \in \Dcaln$, $\|X\|_1$
is always equal to the constant $n$. This fact implies that the $L_1$-norm
regularization does not work for problem \eqref{equ:prob:optperm}. Considering
the good performance of the $L_p$-norm regularization in recovering sparse
solutions \cite{chen2010lower, ji2013beyond, liu2015joint}, we aim to investigate this technique to handle the hard term $\|X\|_0$. Given any $\epsilon \geq 0$ and $0<p<1$, we  use
$h(X) = \|X + \epsilon \1\|_p^p$ to approximate $\|X\|_0$ in
\eqref{equ:prob:qap:l0} and obtain the corresponding $L_p$-norm regularization model:
 \be\label{equ:prob:qap:lp:epsilon}
 \min\limits_{X \in \Dcaln}\  F_{\sigma,p,\epsilon}(X) \coloneqq f(X) +  \sigma \|X + \epsilon {\bf 1}\|_p^p.
 \ee
Note that the problem of minimizing $h(X)$ over $\Dcaln$ and the problem of minimizing $\|X\|_0$ over $\Dcaln$ have the same solution set of permutation matrices. Therefore, $h(X)$ is a good approximation of $\|X\|_0$ and problem \eqref{equ:prob:qap:lp:epsilon} is a good approximation of problem \eqref{equ:prob:optperm}.
The parameter $\sigma$ in problem \eqref{equ:prob:qap:lp:epsilon} mainly
controls the sparsity of $X$ while the parameter $\epsilon$  affects the
smoothness of the regularizer. 
Roughly speaking, problem \eqref{equ:prob:qap:lp:epsilon} with a large $\sigma$
returns a sparse $X$ close to a permutation matrix, while problem
\eqref{equ:prob:qap:lp:epsilon} with a relatively large (not too small) $\epsilon$
is often easier to be solved.

We now show the exactness of the $L_p$-norm regularization \eqref{equ:prob:qap:lp:epsilon}.
 \begin{theorem}\label{theorem:lp:epsilon:exact}
   Suppose that $X_{\sigma, p, \epsilon}$ is  a global solution of problem \eqref{equ:prob:qap:lp:epsilon} with
   \be \label{equ:theorem:lp:epsilon:exact:00}
   \sigma > \sigma^*_{p, \epsilon} \coloneqq \max\left\{\frac{\overline{\nu}_f}{\overline{\nu}_h},0\right\},
   \ee
   where $\overline{\nu}_f$ and $\overline{\nu}_h$ are defined in \eqref{equ:f:concave} and \eqref{equ:nu_h}, respectively. Then $X_{\sigma, p, \epsilon}$ is also a global solution of problem \eqref{equ:prob:optperm}.
 \end{theorem}
 \begin{proof}
By \eqref{equ:lemma:concave:lp} in Proposition \ref{lemma:concave:lp}, for any $\sigma > 0$, there holds
\be\label{equ:theorem:lp:exact:a2}
\sigma h(tX + (1-t) Y) \geq \sigma \Big(th(X) + (1-t) h(Y) + \frac{\overline{\nu}_h}{2} t(1-t)\|Y - X\|_{\Fsf}^2 \Big)
\ee
for any $X, Y \in \Dcaln$ and $0 \leq t \leq 1$. By the assumption \eqref{equ:f:concave}, we always have \eqref{equ:grad:Lips}.
Summing \eqref{equ:theorem:lp:exact:a2} and \eqref{equ:grad:Lips}, and noticing the definition of  $F_{\sigma, p,\epsilon} (\cdot)$,  we conclude that
 $F_{\sigma,p,\epsilon}(X)$ is strongly concave with the positive parameter
 $\sigma  \overline{\nu}_h - \overline{\nu}_f$ over $\Dcaln$. Thus, the global minimum of \eqref{equ:prob:qap:lp:epsilon} must be attained at the vertices of $\Dcaln$, namely, $X_{\sigma, p,\epsilon}$ must be a permutation matrix.

  Since the global solution $X^*$ of \eqref{equ:prob:optperm} is feasible for \eqref{equ:prob:qap:lp:epsilon}, it follows that
\be
 f(X^*) + \sigma \|X^* + \epsilon \1 \|_p^p \geq f(X_{\sigma,p,\epsilon}) + \sigma \|X_{\sigma,p,\epsilon} + \epsilon \1\|_p^p,
\ee
which, together with  $\|X^*\|_p^p = \|X_{\sigma, p,\epsilon}\|_p^p = n(1 +
\epsilon)^p,$ implies  $f(X^*) \geq f(X_{\sigma, p,\epsilon})$. Recalling that
$X_{\sigma, p,\epsilon}$ is a permutation matrix, we see that  $X_{\sigma,
p,\epsilon}$ is also a  global solution of \eqref{equ:prob:optperm}.
 \end{proof}

 Similar to the results in \cite{bian2015complexity, chen2010lower}, we define the KKT point of \eqref{equ:prob:qap:lp:epsilon} as follows.
\begin{definition}
A point $X \in \Dcaln$ is called as a KKT point of \eqref{equ:prob:qap:lp:epsilon} if there exist $S \in \Rbb^{n \times n}$ and $\lambda, \mu \in \Rbb^n$ such that, for any $(i,j) \in \Ncal \times \Ncal$, there holds
\be \label{equ:KKT}
  \begin{cases}
X_{ij}  S_{ij} = 0,\ S_{ij} \geq 0, \\
 \left(W_{ij} - \lambda_i - \mu_j \right)(X_{ij} + \epsilon) + \sigma p ( X_{ij} + \epsilon)^{p} = \epsilon S_{ij},
\end{cases}
\ee
where $W_{ij} = (\nabla f(X))_{ij}$ and  $S$, $\lambda$, and $\mu$ are the  Lagrange multipliers corresponding to the constraints  $X \geq 0$, $X\e = \e,$ and $X^{\Tsf} \e = \e$, respectively.
\end{definition}

We now estimate the lower bound for nonzero elements of the KKT points of problem \eqref{equ:prob:qap:lp:epsilon}.
\begin{theorem}\label{theorem:lowerbound}
  Suppose that $\bar{X}$ is a KKT point of \eqref{equ:prob:qap:lp:epsilon}.
  Then, for any $(i,j) \in \Lambda(\bX)$, we have
  \be\label{equ:lowerbound}
 \bX_{ij}
  \geq \max\left( \bar c- \epsilon,0\right),
  \ee
  where $\bar c = \big(\|\bX\|_0^{1-p}(n+\|\bX\|_0\epsilon)^p - (n - 1)(1 + \epsilon)^{p-1} + \sqrt{2n} \frac{L \sqrt{n} + \|\nabla f(\0)\|_{\Fsf}}{\sigma p}\big)^{\frac{1}{p-1}}$.
\end{theorem}
\begin{proof}
Since $\bX$ is a KKT point of \eqref{equ:prob:qap:lp:epsilon},  it follows from \eqref{equ:KKT} that
\be\label{equ:theorem:lowerbound:a1}
\sum_{i\in \Lambda(\bX_{\cdot j})} \bX_{ij} = 1, \ \forall \ j \in \Ncal, \ \  \sum_{j \in \Lambda(\bX_{i\cdot})} \bX_{ij} = 1, \ \forall \ i \in \Ncal
\ee
and
\be\label{equ:theorem:lowerbound:a2}
\bW_{ij} + \sigma p (\bX_{ij} + \epsilon)^{p-1} = \lambda_i + \mu_j, \quad \forall \ (i,j) \in \Lambda(\bX),
\ee
where $\bW_{ij} = (\nabla f(\bX))_{ij}$.
Multiplying by $\bX_{ij}$ on both sides of \eqref{equ:theorem:lowerbound:a2} and then summing them over $(i,j) \in \Lambda(\bX)$, we have
\be\label{equ:theorem:lowerbound:b1}
\begin{split}
 \sum_{(i,j) \in \Lambda(\bX)} \left(\bW_{ij} \bX_{ij}+ \sigma p (\bX_{ij}+\epsilon)^{p-1}\bX_{ij}  \right)
 = \sum_{(i,j)\in \Lambda(\bX)} (\lambda_i + \mu_j)\bX_{ij} = (\lambda + \mu)^{\Tsf}\e,
\end{split}
\ee

\noindent where the second equality is due to \eqref{equ:theorem:lowerbound:a1} and
\be
\sum_{(i,j)\in \Lambda(\bX)} (\lambda_i + \mu_j)\bX_{ij} = \sum_{i=1}^n \lambda_i\Big(\sum_{j \in \Lambda(\bX_{i \cdot})} \bX_{ij}\Big) + \sum_{j=1}^n \mu_j \Big(\sum_{i \in \Lambda(\bX_{\cdot j})} \bX_{ij}\Big).\nn
\ee
We next claim that one can pick $n$ elements from  different columns
and rows of $\bW_{ij} + \sigma p
(\bX_{ij} + \epsilon)^{p-1}$ with $i,j \in \Lambda(\bX)$ such that their
summation  is
equal to $(\lambda + \mu)^{\Tsf} \e$. To prove this claim, we consider
any $\tX \in \Pi_n$ and $\Lambda(\tX) \subseteq  \Lambda(\bX)$. The same
argument as in the derivation of \eqref{equ:theorem:lowerbound:b1} can be used to show
\be \label{equ:theorem:lowerbound:b2}
\sum_{(i,j) \in \Lambda(\bX)}\big( \bW_{ij} \tX_{ij}+ \sigma p (\bX_{ij} + \epsilon)^{p-1}\tX_{ij} \big)
= \sum_{(i,j)\in \Lambda(\bX)} (\lambda_i + \mu_j)\tX_{ij} = (\lambda + \mu)^{\Tsf}\e.
\ee
Combining \eqref{equ:theorem:lowerbound:b1} and \eqref{equ:theorem:lowerbound:b2} and then rearranging the obtained equation, we have
\be\label{equ:theorem:lowerbound:b3}
\sum_{(i,j) \in \Lambda(\bX)} (\bX_{ij} + \epsilon)^{p-1}\tX_{ij}
=\frac{1}{\sigma p}\sum_{(i,j) \in \Lambda(\bX)} \bW_{ij} (\bX_{ij} - \tX_{ij}) +  \sum_{(i,j) \in \Lambda(\bX)} (\bX_{ij} + \epsilon)^{p-1} \bX_{ij}
\ee
for any $\tX \in \Pi_n$ with $\Lambda(\tX) \subseteq  \Lambda(\bX)$.

Given any $(i_0,j_0) \in \Lambda(\bX)$, we choose some special $\tX \in \Pi_n$
with $\Lambda(\tX) \subseteq  \Lambda(\bX)$ such that $\tX_{i_0j_0} = 1$. Hence,
using $0 < \bX_{ij} \leq 1$, we have
\be\label{equ:theorem:lowerbound:c1}
\sum_{(i,j) \in \Lambda(\bX)} (\bX_{ij} + \epsilon)^{p-1}\tX_{ij} \geq (\bX_{i_0,j_0} + \epsilon)^{p-1} + (n-1)(1 + \epsilon)^{p-1}.
\ee
It is easy to see that $\|\bX - \tX\|_{\Fsf} \leq \sqrt{2n}$. Thus, we obtain
\be\label{equ:theorem:lowerbound:c2}
\sum_{(i,j) \in \Lambda(\bX)} \bW_{ij} (\bX_{ij} - \tX_{ij}) \leq \|\nabla f(\bX)\|_{\Fsf} \|\bX - \tX\|_{\Fsf} \leq \sqrt{2n} (L \sqrt{n} + \|\ \nabla f(\0)\|_{\Fsf}).
\ee
On the other hand, by the concavity of $(z+\epsilon)^p$ and $(1/z + \epsilon)^p z$ on $(0,1]$, we know
\be\label{equ:theorem:lowerbound:c3}
\begin{split}
\sum_{(i,j) \in \Lambda(\bX)} (\bX_{ij} + \epsilon)^{p-1} \bX_{ij}& \leq \sum_{i=1}^n \sum_{j \in \Lambda(\bX_{i\cdot})} (\bX_{ij} + \epsilon)^p
\leq  \sum_{i=1}^n \Big(\frac{1}{\|\bX_{i\cdot}\|_0} +\epsilon \Big)^{p} \|\bX_{i\cdot}\|_0  \\
&  \leq  \left( n + \|\bX\|_0 \epsilon\right)^p \|\bX\|_0^{1-p}.
\end{split}
\ee
Substituting \eqref{equ:theorem:lowerbound:c1},
\eqref{equ:theorem:lowerbound:c2}, and \eqref{equ:theorem:lowerbound:c3} into
\eqref{equ:theorem:lowerbound:b3} and by some easy calculations, we have
\eqref{equ:lowerbound}.
\end{proof}

It is worthwhile to remark that Chen {\it et al.} \cite{chen2010lower} and Lu
\cite{lu2014iterative} established the lower bound theory for the nonzero
elements of the KKT points of unconstrained $L_p$ regularization problems. Based
on our limited knowledge, our lower bound estimate  \eqref{equ:lowerbound}
appears to be a novel  explicit lower bound for the nonzero elements of the
KKT points of $L_p$ regularization problem with linear (equality and inequality)
constraints. 
 It provides a theoretical foundation for rounding the
 approximate optimal solutions of  \eqref{equ:prob:qap:lp:epsilon} to be
 permutation matrices. In practice, we usually only solve problem
 \eqref{equ:prob:qap:lp:epsilon} inexactly (i.e., stop the algorithm early) and
 perform the rounding procedure when  the approximate solution  is near to  some
 permutation matrix. Consequently, the computational cost can be saved.

We next characterize the connections
between the local minimizers of problem
\eqref{equ:prob:qap:lp:epsilon} and the permutation matrices. Our results are
extentions of the properties in \cite{ge2011note} and \cite{liu2016smoothing} 
in the sense that theirs are only for problem \eqref{equ:prob:qap:lp:epsilon} with $\epsilon = 0$.
\begin{theorem}\label{theorem:lp:epsilon:local}
Each permutation matrix is a local minimizer of \eqref{equ:prob:qap:lp:epsilon} with
\be \label{equ:theorem:lp:epsilon:local:a0}
\sigma > \bar \sigma_{p,\epsilon} \coloneqq \frac{c}{p} \cdot \frac{L(2 + \sqrt{n}) + \|\nabla f(\0)\|_{\Fsf}}{\epsilon^{p-1} -(1/2 + \epsilon)^{p-1}},
\ee
where $c >1$ is a constant.
 On the other hand, any local minimizer of \eqref{equ:prob:qap:lp:epsilon} with $\sigma > \sigma_{p,\epsilon}^*$ is a permutation matrix, where $\sigma_{p, \epsilon}^*$ is defined in \eqref{equ:theorem:lp:epsilon:exact:00}.
 \end{theorem}
 \begin{proof}
   Let $\bX$ be a permutation matrix and $X^{l}, l= 1, \ldots, n! - 1$ are all the remaining  permutation matrices.  For any fixed $l$,  consider the feasible direction $D^l = X^l - \bX$. Then $\bX + tD^l \in \Dcaln$  for any $t \in [0,1]$. Denote
\be
\begin{split}
\Lambda_1^l = \{(i,j)\mid \bX_{ij} = 0, \ X^l_{ij} = 0 \}, \  \Lambda_2^l = \{(i,j)\mid \bX_{ij} = 0, \ X^l_{ij} = 1 \},    \\
\Lambda_3^l = \{(i,j)\mid \bX_{ij} = 1, \ X^l_{ij} = 0 \}, \  \Lambda_4^l = \{(i,j)\mid \bX_{ij} = 1, \ X^l_{ij} = 1 \}.
\end{split}
\label{equ:theorem:lp:epsilon:local:Lambda}
\ee
Clearly, we can see that $|\Lambda_1^l |, |\Lambda_4^l | \geq 0$ and
$|\Lambda_2^l | = |\Lambda_3^l | \geq 2$, and $\Ncal \times \Ncal =
\bigcup_{i=1}^4 \Lambda_i^l$. In addition, we also have that $\|D^l\|_{\Fsf}^2 = 2 |\Lambda_2^l | \geq 4.$
By \eqref{equ:grad:Lips0} and the mean-value theorem,  for any $t \in (0,1]$, there exists $\xi \in (0,1)$ such that
\be \label{equ:theorem:lp:epsilon:local:f:a1}
\begin{split}
f(\bX + t D^l) - f(\bX) ={} & t \langle \nabla f(\bX + \xi t D^l), D^l \rangle  \\
\geq{} & - t \|D^l \|_{\Fsf}\|\nabla f(\bX + \xi t D^l)\|_{\Fsf}  \\
\geq{} & -t \|D^l\|_{\Fsf} \cdot \left( L \|\bX + \xi t D^l \|_{\Fsf} + \| \nabla f(\0) \|_{\Fsf}\right) \\
\geq{} & -t  \left(L (\|D^l\|_{\Fsf}^2 + \|D^l\|_{\Fsf}\sqrt{n}) + \|D^l\|_{\Fsf} \| \nabla f(\0) \|_{\Fsf} \right),
\end{split}
\ee
where the last inequality is due to $\|\bar X\|_{\Fsf} = \sqrt{n}$.
Moreover, for any $t \in (0,1/2)$, we have
\be\label{equ:theorem:lp:epsilon:local:h:a3}
\begin{split}
 h(\bX + t D^l) - h(\bX)  ={} & \sum_{(i,j) \in \bigcup_{i=1}^4 \Lambda_i^l} \left( (\bX_{ij} + t D^l_{ij} + \epsilon)^p - (\bX_{ij} + \epsilon)^p \right) \\
={} & | \Lambda_2^l | \left( (t + \epsilon)^p - \epsilon^p \right) +  |\Lambda_3^l |   \left((1 - t + \epsilon)^p - (1 + \epsilon)^p\right)
\\
> {} &\frac{p}{2} \| D^l\|_{\Fsf}^2 \left( (t + \epsilon)^{p-1} -  (1 - t + \epsilon)^{p-1} \right)  t \\
>{} &\frac{p}{2} \| D^l\|_{\Fsf}^2  \left( (t + \epsilon)^{p-1} -  (1/2 + \epsilon)^{p-1} \right)  t,
\end{split}
\ee
where the second equality  is due to
\eqref{equ:theorem:lp:epsilon:local:Lambda};  the first inequality uses
$|\Lambda_2^l | = |\Lambda_3^l | = \frac{\|D^l\|_{\Fsf}^2}{2}$ and the strong
concavity of the function $h(x) = (x + \epsilon)^p$; the second inequality follows from $0 < p  < 1$   and   $t \in (0,1/2]$. Denote $
\tilde t \coloneqq \big(c^{-1} \epsilon^{p-1} + (1 - c^{-1})(1/2 +
\epsilon)^{p-1}\big)^{\frac{1}{p-1}} - \epsilon$. Using the fact that the function $(\xi_1 + \xi_2(1/2 + z)^{p-1})^{\frac{1}{p-1}}$ with $\xi_1, \xi_2 >0$ is  increasing with $z$ in $(0, \infty)$, we have $\tilde t  > \big(c^{-1} \epsilon^{p-1} + (1 - c^{-1})\epsilon^{p-1}\big)^{\frac{1}{p-1}} - \epsilon = 0.$ Furthermore, if we restrict $t \in (0, \bar t\,)$ with $\bar t = \min(1/2, \tilde t\,)$,
we obtain from  \eqref{equ:theorem:lp:epsilon:local:h:a3} that
$$h(\bX + t D^l) - h(\bX)  > \frac{p}{2} \| D^l\|_{\Fsf}^2 \frac{\epsilon^{p-1} - (1/2 + \epsilon)^{p-1}}{c},$$
which together with \eqref{equ:theorem:lp:epsilon:local:a0} and $\|D^l\|_{\Fsf} \geq 2$ implies that
\be\label{equ:theorem:lp:epsilon:local:h:a4}
\sigma( h(\bX + t D^l) - h(\bX) )>
t  \left(L(\| D^l\|_{\Fsf}^2  + \| D^l\|_{\Fsf} \sqrt{n}) +   \| D^l\|_{\Fsf}\|\nabla f(\0)\|_{\Fsf}\right).
\ee
Combining  \eqref{equ:theorem:lp:epsilon:local:f:a1} and \eqref{equ:theorem:lp:epsilon:local:h:a4}, we have
\be \label{equ:theorem:lp:epsilon:local:h:a5}
F_{\sigma,p,\epsilon}(\bX + t D^l) -  F_{\sigma,p,\epsilon}(\bX) >0
\ee
for any $t\in(0,\, \bar t\,)$.
This fact means that  $D^l$, $l = 1, \ldots, n! - 1$, are strictly  increasing feasible directions.
 Let $\mbox{Conv}(\bX, \bar t)$ denote the convex hull spanned by points $\bX$ and $\bX + \bar t D^l$, $l = 1, \ldots, n! - 1$.  Thus for any  $X \in \mbox{Conv}(\bX, \bar t\,)$, we have $F_{\sigma, p, \epsilon}(X) > F_{\sigma, p, \epsilon}(\bX)$ by  \eqref{equ:theorem:lp:epsilon:local:h:a5} and the strict concavity of $h(X)$. Moreover, one can always choose a  sufficiently small but fixed $t >0$ such that $\mathcal{B}(\bX, t) \cap \Dcaln \subset \mbox{Conv}(\bX, \bar t),$  where $\mathcal{B}(\bX, t) = \{X \in \Rbb^{n \times n}\mid \|X - \bar X\|_{\Fsf} \leq t\}$. Consequently, $\bX$ is a local minimizer of problem \eqref{equ:prob:qap:lp:epsilon}.

   On the other hand, if a local minimizer $\bar X$ of
   \eqref{equ:prob:qap:lp:epsilon} with $\sigma > \sigma_{p,\epsilon}^*$  is not
   a permutation matrix, there must exist a feasible direction $D$ such that
   $\bar X + t' D$ and $\bar X - t' D$ both belong to $\Dcaln$ for some
   sufficiently small positive $t'$.  Note that the function $F_{\sigma,
   p,\epsilon}$ is strongly concave when $\sigma > \sigma_{p,\epsilon}^*$. Thus, we must have
   \[\min\left(F_{\sigma, p,\epsilon}(\bar X - t' D), F_{\sigma, p,\epsilon}(\bar X + t' D) \right) < F_{\sigma,p,\epsilon} (\bar X),\] which is a contradiction to the local optimality of $\bX$.  Therefore, $\bX$ must be a permutation matrix. \quad
 \end{proof}

 Theorem \ref{theorem:lp:epsilon:local} on the equivalence between
 the local minimizers of problem
\eqref{equ:prob:qap:lp:epsilon} and the permutation matrices implies
that finding a local minimizer of problem \eqref{equ:prob:qap:lp:epsilon} is
easy but finding the global solution of problem \eqref{equ:prob:qap:lp:epsilon}
is as difficult as finding the global solution of the original NP-hard problem
\eqref{equ:prob:optperm}.  Our extensive numerical experiments show that our algorithm can
 often efficiently identify high quality solutions, in particular, with the help of
 certain techniques to exclude a large portion of non-optimal permutation matrices.

\section{An $L_{p}$ regularization algorithm for \eqref{equ:prob:optperm}} \label{section:lp:alg}
 In this section, we give an $L_p$ regularization algorithmic framework and its
 practical version for solving problem \eqref{equ:prob:optperm}. Since the
 projection onto the set of doubly stochastic matrices needs to be computed many
 times in our proposed algorithm, we also propose a fast dual gradient method
 for solving it.

\subsection{An $L_p$ regularization algorithmic framework}\label{subsection:lp:alg:frame}

Theorem \ref{theorem:lp:epsilon:exact} implies that, for any $\epsilon >0$,   we
can choose a fixed $\sigma$ larger than $\sigma_{p,\epsilon}^*$ to solve
\eqref{equ:prob:qap:lp:epsilon} and obtain a permutation matrix for problem
\eqref{equ:prob:optperm}.  In practical implementation, it might be better to
dynamically increase $\sigma$ from an initial value $\sigma_0$ since a large
$\sigma$ may result in the ill-conditioned subproblem. Since
$\sigma_{p,\epsilon}^*$ is increasing with respect to $\epsilon$ if
$\overline{\nu}_f > 0$ (see Theorem \ref{theorem:lp:epsilon:exact}),
dynamically decreasing $\epsilon$ from a relatively large value $\epsilon_0$ is
helpful in finding a permutation matrix relatively faster.

To summarize, we solve a sequence of \eqref{equ:prob:qap:lp:epsilon} with strictly increasing
parameters $\{\sigma_k \mid k = 0,1, 2, \ldots \}$ and decreasing parameters $\{\epsilon_k
\mid k =0, 1, 2, \ldots\}$. For a fixed $k$, we denote $X_{k}^{(0)}$ as the starting point
for solving \eqref{equ:prob:qap:lp:epsilon} with $\sigma_{k}$ and $\epsilon_k$.
After computing a KKT point $X_k$ for \eqref{equ:prob:qap:lp:epsilon} with
$\sigma_{k}$ and $\epsilon_k$, we use a warm start technique to set
\be \label{equ:warmstart1}
X_{k+1}^{(0)} \coloneqq X_{k}
\ee
if $X_k$ is not a KKT point of  problem  \eqref{equ:prob:qap:lp:epsilon}  with $\sigma_{k+1}$ and $\epsilon_{k+1}$.
Otherwise,  we set
 \be \label{equ:warmstart2}
X_{k+1}^{(0)}  =
\mbox{a perturbation of}   \ X_k\ \mbox{in}\ \Dcaln.
\ee
The following theorem shows that once we find a permutation matrix, which is a KKT point of problem \eqref{equ:prob:qap:lp:epsilon}  with some $\sigma_k$ and $\epsilon_k$, it is reasonable to stop the iterative procedure.
    \begin{theorem}\label{theorem:lp:epsilon:KKT2}
   If a permutation matrix $X$ is a KKT point of  \eqref{equ:prob:qap:lp:epsilon} with fixed $\hat \sigma$ and $\hat \epsilon > 0$, then it is also a KKT point of \eqref{equ:prob:qap:lp:epsilon} with $\sigma$ and $\epsilon$ satisfying  $\sigma \geq \hat \sigma$ and $0< \epsilon \leq \hat \epsilon$.
 \end{theorem}

 \begin{proof}
   Denote the corresponding permutation vector of $X$ as $\pi$. Let the corresponding multipliers be $\lambda^{\hat \sigma, \hat \epsilon}$, $\mu^{\hat \sigma, \hat \epsilon},$ and $S^{\hat \sigma, \hat \epsilon}$. Then we know from  \eqref{equ:KKT} that  for   $i, j\in \Ncal$, there holds
 \begin{equation}\label{equ:theorem:lp:epsilon:KKT2:Shat}
S_{ij}^{\hat \sigma, \hat \epsilon} = \begin{cases}
   W_{ij} - \lambda_i^{\hat \sigma, \hat \epsilon} - \mu_j^{\hat \sigma, \hat \epsilon} +  \hat \sigma p(1 + \hat \epsilon)^{p-1}, & i = \pi(j),\\
   W_{ij} - \lambda_i^{\hat \sigma, \hat \epsilon} - \mu_j^{\hat \sigma, \hat \epsilon} +  \hat \sigma p\hat \epsilon^{p-1}, & i \neq \pi(j),
  \end{cases}
  \end{equation}
  and
  $S_{ij}^{\hat \sigma, \hat \epsilon} = 0$ if $i = \pi(j)$ and $S_{ij}^{\hat \sigma, \hat \epsilon}  \geq 0$ if $i \neq \pi(j)$.
 For any $\sigma \geq  \hat \sigma$ and $0 < \epsilon \leq \hat \epsilon$, define
   \be
   \lambda^{\sigma,  \epsilon} = \lambda^{\hat \sigma, \hat \epsilon} + p\left(\sigma  ( 1+ \epsilon)^{p-1} - \hat \sigma (1+ \hat \epsilon)^{p-1}\right) \mathbf{e}, \quad
   \mu^{\sigma,  \epsilon} = \mu^{\hat \sigma, \hat \epsilon} \nn
   \ee
  and
  \begin{equation}\label{equ:theorem:lp:epsilon:KKT2:S}
S_{ij}^{\sigma, \epsilon} =
\begin{cases}
W_{ij} - \lambda_i^{\sigma, \epsilon} - \mu_j^{\sigma, \epsilon} +  \sigma p(1 + \epsilon)^{p-1}, & i = \pi(j),\\
W_{ij} - \lambda_i^{\sigma, \epsilon} - \mu_j^{\sigma, \epsilon} +  \sigma p\epsilon^{p-1}, & i \neq \pi(j).
  \end{cases}
  \end{equation}
  We now show that $S^{\sigma, \epsilon}$, $\lambda^{\sigma,  \epsilon},$ and $\mu^{\sigma,  \epsilon}$ satisfy KKT condition \eqref{equ:KKT}.  For any  $i = \pi(j)$,  we have from \eqref{equ:theorem:lp:epsilon:KKT2:Shat} and \eqref{equ:theorem:lp:epsilon:KKT2:S} that
  \begin{align}\label{equ:theorem:lp:epsilon:KKT2:a1}
S_{ij}^{\sigma, \epsilon} ={} &  W_{ij} - \lambda_i^{\sigma,  \epsilon} - \mu_{j}^{\sigma,  \epsilon} + \sigma p (1 + \epsilon)^{p-1}
=S_{ij}^{\hat \sigma, \hat \epsilon}= 0;
\end{align}
for any $i \neq \pi(j)$, we have
\be\label{equ:theorem:lp:epsilon:KKT2:a2}
\begin{split}
  S_{ij}^{\sigma, \epsilon} ={} & W_{ij} - \lambda_i^{\sigma, \epsilon} - \mu_j^{\sigma, \epsilon} + \sigma p \epsilon^{p-1} \\
  ={} &  S_{ij}^{\hat \sigma, \hat \epsilon} + p\hat \sigma\left(( 1 + \hat \epsilon)^{p-1} - \hat \epsilon^{p-1} \right)  - p\sigma\left(   ( 1 + \epsilon)^{p-1} -  \epsilon^{p-1} \right)  \\
  \geq{} &S_{ij}^{\hat \sigma, \hat \epsilon} + p(\hat \sigma - \sigma )  \left(   ( 1 + \epsilon)^{p-1} -  \epsilon^{p-1} \right)
  \geq{} S_{ij}^{\hat \sigma, \hat \epsilon} \geq 0,
\end{split}
\ee
where the first inequality is due to that $(1+ z)^{p-1} - z^{p-1}$ is increasing in $z$ in $(0, \infty)$.
Combining \eqref{equ:theorem:lp:epsilon:KKT2:a1} and \eqref{equ:theorem:lp:epsilon:KKT2:a2}, we can see that $X$ is also a KKT point of \eqref{equ:prob:qap:lp:epsilon} with $\sigma \geq \hat \sigma$ and $0 < \epsilon \leq \hat \epsilon$.\quad
 \end{proof}

The $L_p$ regularization algorithmic framework for solving  problem
\eqref{equ:prob:optperm} is outlined in Algorithm \ref{alg:lp:abstract}.

 \begin{algorithm}[H]
  Given $X_0 \in \Dcaln$, set $k=0$, $\epsilon_0 >0$,  $\sigma_0 > 0$, $\hat c
  >0$. \\
     \While{$\|X_k\|_p^p > n$}{
     Set $X_{k}^{(0)}$ according to \eqref{equ:warmstart1} and \eqref{equ:warmstart2}.  \\
   Find a KKT point  $X_{k}$ of problem \eqref{equ:prob:qap:lp:epsilon} with $\sigma_k$ and $\epsilon_k$ staring from $X_{k}^{(0)}$ such that $F_{\sigma_k, p, \epsilon_k} (X_{k}) \leq F_{\sigma_k, p, \epsilon_k} (X_k^{(0)})$.\\
  Choose $\sigma_{k+1} > \sigma_k + \hat c$, $0< \epsilon_{k+1} \leq \epsilon_k$  and  set $k = k + 1$.  }
     \caption{An $L_p$ regularization algorithmic framework for problem \eqref{equ:prob:optperm}}\label{alg:lp:abstract}
  \end{algorithm}

  The convergence of Algorithm \ref{alg:lp:abstract} is summarized as follows.
  \begin{theorem} \label{thm:convergence}
   Algorithm \ref{alg:lp:abstract} returns a permutation matrix in a finite
   number of iterations.
 \end{theorem}
 \begin{proof}
By the update scheme of  $\sigma_k$ and $\epsilon_k$, we have $\sigma_k >
\max\{\sigma_{p,0}^*, \bar \sigma_{p,\epsilon_0}\}$ for some finite $k$, where
$\sigma_{p,\epsilon}^*$ and $\bar \sigma_{p,\epsilon}$ are defined in
\eqref{equ:theorem:lp:epsilon:exact:00} and
\eqref{equ:theorem:lp:epsilon:local:a0}, respectively. The update scheme of
$\epsilon_k$ yields that $\sigma_{p,0}^* > \sigma_{p,\epsilon_k}^*$ and $\bar \sigma_{p, \epsilon_0} \geq \bar \sigma_{p, \epsilon_k}$.  It follows from the proof of Theorem \ref{theorem:lp:epsilon:exact} that $F_{\sigma_k, p, \epsilon_k}(X)$ with $\sigma_k > \max\{\sigma_{p,0}^*, \bar \sigma_{p,\epsilon_0}\}$ and $\epsilon_k > 0$ is strongly concave. Consequently, any KKT point $X_k$ of problem \eqref{equ:prob:qap:lp:epsilon} with $\sigma >  \max\{\sigma_{p,0}^*, \bar \sigma_{p,\epsilon_0}\}$ and $0< \epsilon \leq  \epsilon_0$ is either a local minimizer (which must be a permutation matrix) or the unique global maximizer of problem \eqref{equ:prob:qap:lp:epsilon}. Consider the case when  $X_k$ is  the unique global maximizer of problem \eqref{equ:prob:qap:lp:epsilon} with $(\sigma,\epsilon)=(\sigma_{k+1},\epsilon_{k+1})$.
By the choice of  $X_{k+1}^{(0)}$, we  have $F_{\sigma_{k+1},
p, \epsilon_{k+1}}(X_{k+1}^{(0)}) < F_{\sigma_{k+1}, p, \epsilon_{k+1}}(X_k)$.
This inequality together with the fact that
$F_{\sigma_{k+1}, p, \epsilon_{k+1}} (X_{k+1}) \leq F_{\sigma_{k+1}, p, \epsilon_{k+1}} (X_{k+1}^{(0)})$ from Line 4 in Algorithm  \ref{alg:lp:abstract}
implies that the KKT point $X_{k+1}$ must be a local minimizer of problem
\eqref{equ:prob:qap:lp:epsilon}, which means that $X_{k+1}$ is a permutation matrix.
In the other case, we choose  $X_{k+1}^{(0)} = X_k$ by \eqref{equ:warmstart1}.
A similar argument shows that $X_{k+1}$ is a permutation matrix.
  \end{proof}
\subsection{A practical $L_p$ regularization algorithm}
\label{subsection:lp:alg:practical} We first use the projected
gradient method with the BB step sizes
\cite{barzilai1988two, birgin2000nonmonotone}  to compute an approximate KKT
point of problem \eqref{equ:prob:qap:lp:epsilon}.   Specifically, starting from
initial $X_k^{(0)}$, the projected gradient method for solving problem
\eqref{equ:prob:qap:lp:epsilon} with $\sigma_k$ and $\epsilon_k$ iterates
 as follows:
 \be \label{equ:IRL1:X2}
 X_{k}^{(i+1)} = X_k^{(i)} + \delta^j D^{(i)}, \quad \delta \in (0,1),
 \ee
 where the search direction $D^{(i)}= \proj_{\Dcaln} \big(X_k^{(i)} - \alpha_i
 \nabla F_{\sigma_k, p, \epsilon_k} (X_k^{(i)})\big) - X_k^{(i)}$ and
 $\proj_{\Dcaln}(\cdot)$ is the projection onto $\Dcaln$. The step size
 $\alpha_i$  is set to be the alternative usage of the large and short BB steps
 \cite{dai2005projected, wen2013feasible}. The parameter $j$ is the smallest nonnegative integer satisfying the nonmonotone line search condition:
\be\label{equ:nmls}
  F_{\sigma_k,p,\epsilon_k}(X_k^{(i)} + \delta^j D^{(i)}) \leq C_i+ \theta \delta^j \langle \nabla F_{\sigma_k, p, \epsilon_k}(X_k^{(i)}), D^{(i)}\rangle, \quad  \theta \in (0,1),
\ee
where  the reference function value $C_{i+1}$ is updated as the convex
combination of $C_i$ and $F_{\sigma_k, p, \epsilon_k}(X_k^{(i)})$, i.e.,
$C_{i+1} =  (\eta Q_i C_i + F_{\sigma_k, p,\epsilon_k} (X_{k}^{(i)}))/Q_{i+1}$,
where  $Q_{i+1} = \eta Q_i + 1$,  $\eta = 0.85$, $C_0 = F_{\sigma_k, p,
\epsilon_k}(X_k^{(0)})$ and $Q_0 = 1$. See  \cite{zhang2004nonmonotone} for more
detailed information.

Any intermediate point $X \in \Dcaln$ can be converted to a permutation matrix via rounding or a greedy procedure, which can be further improved by performing local search. The specific algorithm will be introduced later in Section
\ref{subsection:impDetails}. Other suitable greedy approaches can be adopted as
well. The greedy procedure not only offers a better
permutation matrix, but also guides the update of the parameter $\epsilon$.
Specifically, for any point $X_k^{(i)},$ we generate $\hat X_k^{(i)}$ from $X_k^{(i)}$
 such that $\hat X_k^{(i)} \in \Ncal_2^{\best} (\hat X_k^{(i)})$ and update
$$f_k^{\best} = \displaystyle\min_i \{f(\hat X_k^{(i)})\}~\text{and}~X_k^{\best}
= \arg\displaystyle\min_i \{f(\hat X_k^{(i)})\}.$$ Let $f^{\best}$ be the best
function value among  $f^{\best}_1, \ldots, f^{\best}_{k-1}$. If $f^{\best}_k$
is less than or equal to $f^{\best}$, we set $\epsilon_k=\epsilon_{k-1}$;
otherwise we set $\epsilon_k = \gamma \epsilon_{k-1}$ with $\gamma \in (0,1)$.
Moreover, recall that $\bar \sigma_{p,\epsilon}$ in Theorem
\ref{theorem:lp:epsilon:local} is increasing with respect to $\epsilon$ and
$\lim_{\epsilon \rightarrow 0} \bar \sigma_{p,\epsilon} = 0$. Therefore, if
$\epsilon_k$ is too small, then nearly all of the permutation matrices are local
minimizers of \eqref{equ:prob:qap:lp:epsilon}. This might lead to finding a bad
permutation matrix (in terms of solving problem \eqref{equ:prob:optperm}) with a
high probability.  To avoid this drawback, we propose to set a safeguard $\epsilon_{\min}$ for $\epsilon_k$.  We summarize the strategy of updating $\epsilon_k$ as follows
\be \label{equ:epsilon:update}
\begin{array}{l@{\hspace{1mm}}l}
\mbox{if}& f^{\best}_k < f^{\best} \\
&f^{\best} = f^{\best}_k, X^{\best} = X_{k}^{\best}, \epsilon_k=\epsilon_{k-1},\\
\mbox{else}& \\
& \epsilon_{k} = \max\{\gamma \epsilon_{k-1}, \epsilon_{\min}\}.\\
\mbox{end}
\end{array}
\ee

Finally, we use the continuation scheme to dynamically
update $\sigma$ such that problem \eqref{equ:prob:qap:lp:epsilon} changes from
being strongly convex to being strongly concave.  To achieve this goal, we choose $\sigma_0 =
 \min\left\{\frac{\underline{\nu}_f}{p(1-p)} \epsilon_0^{2-p}, \, \, \sigma_{-} \right\}$,
where $\underline{\nu}_f$ is defined in \eqref{equ:f:concave} and $\sigma_{-}  <
0$ is given. Motivated by the strategy used in \cite{xia2010efficient}, we update $\sigma_{k+1}$ by
\be\label{equ:sigma:update}
 \sigma_{k+1} \coloneqq \min\{\tilde\sigma_{k+1}, \sigma_{\max}\}\ \mbox{and}\  \tilde\sigma_{k+1} =
\begin{cases}
\frac12 \sigma_k, & \sigma_{k} \le \sigma_{-}, \\
0, & \sigma_{-} < \sigma_{k} < 0, \\
\sigma_{+}, & \sigma_k = 0,\\
2 \sigma_k, & \sigma_k \ge \sigma_{+},
 \end{cases}
  \ee
  where $\sigma_{+} = -2^{-l}\sigma_0$, $ l = \lceil \log_2(-\sigma_0)\rceil$ and $\sigma_{\max} >0$ is a safeguard for $\sigma$.

 We summarize the practical $L_{p}$ regularization algorithm for solving problem
 \eqref{equ:prob:optperm} in Algorithm \ref{alg:lp:practical}.
 Specifically, the stopping rules on line 5 are $\tol_i^x = \|X_k^{(i)} - X_k^{(i-1)}\|_{\Fsf}/\sqrt{n}$, $\tol_i^f = \frac{\big|F_{\sigma_k, p,\epsilon_k}(X_k^{(i)}) - F_{\sigma_k, p, \epsilon_k}(X_k^{(i-1)})\big|}{1+ \big| F_{\sigma_k, p, \epsilon_k}(X_k^{(i-1)})\big|}$ and $
  \tau_k^x = \max\left\{\tau_0^x/k^3, \ \tau_{\min}^x\right\}, \tau_k^f = \max\{\tau_0^f/k^3, \ \tau_{\min}^f\},$ where $\tau_{\min}^x, \tau_{\min}^f >0$ are some small parameters.
 Of course, other types of criteria can also be applied as well. As discussed after the proof of Theorem \ref{theorem:lowerbound}, the lower bound  \eqref{equ:lowerbound} may be helpful in Line 9.

 \begin{algorithm}[H]
  Given $X_{0} \in \Dcaln$, set $\epsilon_0, \tau_0^x, \tau_0^f, \tau_{\min}^x, \tau_{\min}^f>0$, $\sigma_0 \leq 0$, $\theta, \delta, \gamma, \tol \in (0,1)$.\\
  Set $k = 0$, $f^{\best} = \infty$.\\
   \While{$\|X_k\|_p^p /n - 1> \tol$}{
   \tcp{Lines 4 - 9: compute an approximate KKT point  of \eqref{equ:prob:qap:lp:epsilon}}
   Choose $X_{k}^{(0)}$ by \eqref{equ:warmstart1} and \eqref{equ:warmstart2}. Set $i = 0$, $\alpha_i > 0$, $\tol_i^x = \tol_i^f = \infty$. \\
 \While{$ \tol_i^x > \tau_k^x$ \textnormal{or} $\tol_i^f > \tau_k^f$ }{
  Compute $D^{(i)} = \proj_{\Dcaln} \big(X_k^{(i)} - \alpha_i \nabla F_{\sigma_k, p, \epsilon_k} (X_k^{(i)})\big) - X_k^{(i)}$.\\
 Find the smallest $j$ such that  $\delta^j$ satisfies \eqref{equ:nmls}.\\
  Set $X_{k}^{(i+1)} = X_k^{(i)} + \delta^j D_k^{(i)}$ and $i = i + 1$. \\
 Compute $\hat X_k^{(i)} \in \Ncal_2^{\best} (\hat X_k^{(i)})$ starting from $X_k^{(i)}$  and update $f_k^{\best}$.
 }
   Update $\epsilon_{k+1},$ $\sigma_{k+1}$, $f^{\best},$ and $X^{\best}$ by \eqref{equ:epsilon:update} and \eqref{equ:sigma:update}.\\
    Set $X_{k+1}=X_{k}^{(i)}$ and $k = k + 1$.
      }
     \caption{A practical $L_{p}$  regularization algorithm for problem \eqref{equ:prob:optperm}. }\label{alg:lp:practical}
  \end{algorithm}

 \subsection{Fast dual gradient algorithm for computing projection onto $\Dcaln$} \label{subsection:dualBB:proj}
Note that the dominant computational cost of Algorithm \ref{alg:lp:practical} is to compute projections onto the set $\Dcaln$.
In this subsection, we propose a fast dual gradient method for solving the projection problem
\be\label{equ:prob:proj}
 \min_{X \in \Rbb^{n\times n}}\, \frac12 \|X - C\|_{\Fsf}^2 \quad \st \quad X\e = \e, \ X^{\Tsf} \e = \e,\  X \geq \0,
 \ee where $C \in \Rbb^{n\times n}$ is given.

 The Lagrangian dual problem of \eqref{equ:prob:proj}  is
\be \label{equ:prob:proj:dual0}
\max_{y, z}\, \min_{X \geq 0} \, \Lcal(X, y,z),
\ee
where
$
 \Lcal(X, y,z) = \frac12 \|X - C\|_{\Fsf}^2  - \langle y, X \e - \e\rangle - \langle z, X^{\Tsf} \e - \e \rangle,\nn
 $
 in which $y, z \in \Rbb^n$ are the Lagrange multipliers associated with the linear constraints $X\e = \e$ and $X^{\Tsf} \e = \e,$ respectively.
 Let $\proj_{+}(\cdot)$ denote the projection onto the nonnegative orthant.
Then the dual problem \eqref{equ:prob:proj:dual0} can be equivalently rewritten as \be \label{equ:prob:proj:dual}
\min_{y, z} \,  \theta(y,z) \coloneqq \frac12 \|\proj_{+} \left(C + y \e^{\Tsf} + \e z^{\Tsf}\right)\|_{\Fsf}^2 -  \langle y + z, \e \rangle.
\ee It can be verified that $\theta(y,z)$ is convex and continuously differentiable and 
\be
\nabla \theta(y,z) = \left[ \begin{array}{l} \proj_{+} \left(C + y \e^{\Tsf} + \e z^{\Tsf} \right)\e - \e \\[4pt] \proj_{+} \left(C + y \e^{\Tsf} + \e z^{\Tsf} \right)^{\Tsf}\e - \e \end{array}\right].
\nn
\ee
Note that   \cite{bai2007computing}  computed  a nearest doubly stochastic matrix with a fixed entry via solving a  dual problem similar to \eqref{equ:prob:proj:dual},  by the nonsmooth Newton method.

We propose to use the gradient method using the BB step sizes to solve the
unconstrained problem \eqref{equ:prob:proj:dual}. The method is named as ``dualBB".
After obtaining the solution  $y^*$ and $z^*$  of \eqref{equ:prob:proj:dual}, we recover the projection of $C$ as
$\proj_{\Dcaln}(C) =  \proj_{+} \left(C + y^*\e^{\Tsf} + \e (z^*)^{\Tsf} \right).$
Our preliminary numerical results on the QAPLIB demonstrate that our proposed dualBB method is faster than
MOSEK in many cases.

 \section{Enhanced variants of $L_p$ regularization algorithm}
 In this section, we combine $L_p$ regularization Algorithm
 \ref{alg:lp:practical} with two techniques to further enhance its performance.
  The first variant uses the cutting plane technique and the second one is
 based on a novel $L_p$ regularization model, which has an additional negative
 proximal point term compared with the model \eqref{equ:prob:qap:lp:epsilon}.

    \subsection{$L_p$ regularization algorithm with cutting plane technique} \label{subsection:lp:cp}
    Let  $\tilde{X}$ be a permutation matrix and let $\tilde{f} = f(\tilde{X})$.
    Assume that $\tilde X$ is not optimal. Then there must exist a constant $c_1 > 0$ such that $f^* \leq \tilde{f} - c_1$. Next, we construct some cuts to shrink the feasible region $\Pi_n$ of problem \eqref{equ:prob:optperm} based on the information of $\tilde{X}$.

 Denote $F_{\omega}(X) \coloneqq f(X) + \omega \|X\|_{\Fsf}^2$, where the parameter $\omega$ is chosen  such that $F_{\omega}$  is strongly convex.
 The ideal cut
  \be\label{equ:quadcut}
    \mbox{QC}(\tilde{X}): \quad F_{\omega}(X) \leq \tilde{f} - c_1 + \omega n
  \ee
is hard to be handled because of the nonlinear
  function $f(X)$.
The first order approximation
\be\label{euq:linear_cut_1}
 \mbox{LC1($\tilde{X}$):} \quad \quad  F_{\omega}(\tilde{X}) + \langle \nabla F_{\omega}(\tilde{X}) , X - \tilde{X} \rangle  \leq \tilde{f}  - c_1 + \omega n.
\ee
is easy but it  might not be tight.

Suppose that $\tilde{X}$ is the best point in its 2-neighborhood, a linear cut
is proposed in \cite{bazaraa1982use} as
\be \label{euq:linear_cut_2}
 \mbox{LC2($\tilde{X}$):} \quad \quad  \langle   \tilde{X},  X  \rangle \leq n - 3,
\ee
which precisely cuts $\frac{n(n-1)}{2} + 1$ permutation matrices from $\Pi_n$.

By using the above two cuts \eqref{euq:linear_cut_1} and \eqref{euq:linear_cut_2}, we can shrink $\Pi_n$ to a strictly smaller set $\Acal_{\Pi_n} := \Pi_n \cap \mbox{LC1}(\tilde X) \cap \mbox{LC2}(\tilde X),$ which still contains the global solution of problem \eqref{equ:prob:optperm}.   Therefore, we can solve the following problem (instead of problem \eqref{equ:prob:optperm}):
   \be \label{equ:prob:optperm:cut}
    \min\limits_{X \in \Acal_{\Pi_n}} \, f(X).
 \ee

By  replacing $\Dcaln$ by $\Acal \coloneqq \Dcaln \cap \mbox{LC1}(\tilde X) \cap \mbox{LC2}(\tilde X) $,  Algorithm \ref{alg:lp:practical} can be employed to solve problem \eqref{equ:prob:optperm:cut}, where the $L_p$ regularization subproblem becomes
\be\label{equ:prob:qap:lp:epsilon:cut}
 \min\limits_{X \in \mathcal{A}}\  F_{\sigma,p,\epsilon}(X).
 \ee

The $L_p$ regularization algorithm  with the cutting place technique ($L_p$-CP) is given as follows.
~\\
     \begin{algorithm}[H]
  Given $\Acal = \Dcaln$, set $K_{\max} > 0$ and $K = 0$.\\
 \While{$K < K_{\max}$}{
  Solve \eqref{equ:prob:optperm:cut} by Algorithm \ref{alg:lp:practical} with  subproblem \eqref{equ:prob:qap:lp:epsilon} there replaced by \eqref{equ:prob:qap:lp:epsilon:cut}, and obtain $X^{\best}_K$.\\
Set $\Acal =  \Acal \cap \left(\mbox{LC1}(X_K^{\best}) \cap \mbox{LC2}(X_K^{\best}) \right)$.\\
 Set  $X^{\best} = \argmin_{K} \{f(X_K^{\best})\}$ and $K = K + 1$.
}
   \caption{An $L_p$-CP algorithm for problem \eqref{equ:prob:optperm}. }\label{alg:lp:cp}
  \end{algorithm}
The solution quality of Algorithm \ref{alg:lp:cp} is generally better
than that of Algorithm \ref{alg:lp:practical} due to the cutting plane
technique. However, Algorithm \ref{alg:lp:cp} might be slower than
Algorithm \ref{alg:lp:practical} because computing the projections onto
the restricted domain $\mathcal{A}$ in Algorithm \ref{alg:lp:cp} is slightly
time consuming than computing the projections onto the set $\Dcaln$ and
computational cost of the projections is the dominant part in both algorithms. Note that Algorithm \ref{alg:lp:cp} with $K_{\max} = 1$ reduces to Algorithm \ref{alg:lp:practical}.

    \subsection{$L_p$ regularization with negative proximal point technique} \label{subsection:lp:negProx}
The cutting plane technique may not improve the performance of Algorithm
\ref{alg:lp:practical}. For instance, Algorithm \ref{alg:lp:cp} with $K_{\max} =
4$ and Algorithm \ref{alg:lp:practical} find the same permutation matrix when
they are used to solve the problem instance ``chr20c'' from QAPLIB. By
investigating the trajectories of the permutation matrices generated by
Algorithm \ref{alg:lp:cp}, we find that the trajectories (corresponding to different $K$) are nearly the same for this problem.  To avoid this  phenomenon,  we need to push the trajectories of the permutation
matrices corresponding to different $K$ far away from each other. It can be
achieved by
adding a negative proximal point term to $f(X)$.

    Given $\mu > 0$ and $\hat X \in \Dcaln,$ consider problem
      \be\label{equ:prob:optperm:negProx}
 \min\limits_{X \in \Pi_n} \   f(X)  - \mu \|X -  \hat X\|_{\Fsf}^2.
 \ee The next theorem shows that the solution of \eqref{equ:prob:optperm:negProx} with an appropriate choice of $\mu$ is a solution of \eqref{equ:prob:optperm}. Moreover, the solution of \eqref{equ:prob:optperm:negProx} is the farthest point away from the given point $\hat X$ among the solution set of \eqref{equ:prob:optperm}.

 \begin{theorem} \label{thm:negProx}
Let $c_2>0$ be such that
\be \label{equ:fxfy}
|f(X) - f(Y)| \geq c_2,\quad \forall~X, Y \in \Pi_n~\text{with} ~f(X) \ne f(Y).
\ee
Suppose that $0 < \mu < \frac{c_2}{2{n}}.$ Then any solution $X_{\mu}^*$ of \eqref{equ:prob:optperm:negProx} is also a solution
of \eqref{equ:prob:optperm}. Moreover, for any solution $X^*$ of \eqref{equ:prob:optperm}, we have
\be \label{equ:thm:negProx:00}
 \|X_{\mu}^* - \hat X\|_{\Fsf} \geq  \|X^* - \hat X \|_{\Fsf}.
\ee

\end{theorem}

\begin{proof}
From the optimality of $X^*_{\mu}$, we obtain
\be \label{equ:theorem:negProx:a1}
f(X_{\mu}^*) - \mu \|X_{\mu}^* - \hat X\|_{\Fsf}^2 \leq f(X^*) - \mu \|X^* - \hat X \|_{\Fsf}^2,
\ee
which is equivalent to
\begin{align}
f(X_{\mu}^*) \leq {}f(X^*) + \mu\left(\|X_{\mu}^*\|_{\Fsf}^2-\|X^*\|_{\Fsf}^2+2\langle X^* - X^*_{\mu}, \hat X \rangle\right).
\nn
\end{align}
 Combining the above inequality with the facts that $\|X_{\mu}^*\|_{\Fsf}=\|X^*\|_{\Fsf}=\sqrt{n},$ $\langle X^*, \hat X \rangle \leq n,$ and $\langle X_{\mu}^*, \hat X \rangle \geq 0$ yields \be\label{equ:theorem:negProx:a3}
f(X_{\mu}^*) \leq   f(X^*) + 2 \mu n < f(X^*) + c_2,
\ee
 where the last inequality is due to the choice of $\mu$. Moreover, it follows from the optimality of $X^*$ that  $f(X^*) \leq f(X_{\mu}^*),$ which, together with \eqref{equ:fxfy} and \eqref{equ:theorem:negProx:a3}, implies $f(X_{\mu}^*) = f(X^*).$ This shows that $X_{\mu}^*$ is a solution to  \eqref{equ:prob:optperm}. From \eqref{equ:theorem:negProx:a1} and $f(X_{\mu}^*) = f(X^*),$ we immediately obtain the desired result \eqref{equ:thm:negProx:00}.
 \end{proof}

By setting $\hat X=\frac{1}{K}\sum_{i = 1}^K \hat X_i$ in Theorem \ref{thm:negProx}, where $\hat X_i\in \Dcaln$ for all $i = 1, \ldots, K,$ and combining it with Theorem \ref{theorem:lp:epsilon:exact}, we obtain the following theorem.
 \begin{theorem}\label{theorem:lp:epsilon:exact:negProx}
   Suppose $0< \mu <\frac{c_2}{2n}$ and $\sigma > \max((\overline{\nu}_f - 2 \mu)/{\overline{\nu}_h },0)$. Then any global solution of
   \be\label{equ:prob:qap:lp:epsilon:cut:2}
 \min\limits_{X \in \Dcaln} \   f(X) +  \sigma \|X + \epsilon {\bf 1}\|_p^p - \mu  \Big\|X - \frac{1}{K}{\sum_{i = 1}^K \hat X_i}\Big\|_{\Fsf}^2
 \ee is also a global solution of problem \eqref{equ:prob:optperm}. Moreover, it is one of the farthest solutions of \eqref{equ:prob:optperm} away from $\frac{1}{K}\sum_{i = 1}^K \hat X_i$.
 \end{theorem}

Based on Theorem \ref{theorem:lp:epsilon:exact:negProx}, we present the $L_p$
regularization algorithm with the negative proximal point technique, which is
dubbed as $L_p$-negProx.

   \begin{algorithm}[H]
  Set $\mu>0$, $K_{\max} > 0$ and $K = 0$. Define $X_K^{\best} = \0$. \\
 \While{$K < K_{\max}$ \textnormal{and} $X_K^{\best} \not \in \{X_1^{\best}, \ldots, X_{K-1}^{\best}\}$}{
  Solve \eqref{equ:prob:optperm:negProx} with $\hat X=\frac{1}{K}{\sum_{i = 1}^K  X_i^{\best}}$ by Algorithm \ref{alg:lp:practical} with  subproblem \eqref{equ:prob:qap:lp:epsilon}  replaced by \eqref{equ:prob:qap:lp:epsilon:cut:2}, and obtain $X^{\best}_K$.\\
Set $\mu= \mu/2,$ $X^{\best}= \argmin_{K} \{f(X_K^{\best})\}$, and $K= K + 1$.
}
   \caption{An $L_p$-negProx algorithm for problem \eqref{equ:prob:optperm}.}\label{alg:lp:negProx}
     \end{algorithm}

Consider the problem ``chr20c'' in QAPLIB again. By simply setting $\mu \equiv
0.1$,  Algorithm \ref{alg:lp:negProx} is able to find the global solution. More
detailed numerical results are reported in Section \ref{subsection:negProx} to
demonstrate the effectiveness of the proposed negative proximal point technique.

 We can also combine Algorithm \ref{alg:lp:practical} with both the cutting
 plane technique and the negative proximal point technique. The resulting
 algorithm is named as $L_p$-CP-negProx and it is described as follows.

    \begin{algorithm}[H]
  Given $\Acal \coloneqq \Dcaln$, set $\mu>0$, $K_{\max} > 0$ and $K = 0$. Define $X_K^{\best} = \0$.\\
 \While{$K < K_{\max}$ \textnormal{and} $X_K^{\best} \not \in \{X_1^{\best}, \ldots, X_{K-1}^{\best}\}$}{
  Solve \eqref{equ:prob:optperm:negProx} with $\hat X=\frac{1}{K}{\sum_{i = 1}^K \hat X_i^{\best}}$ by Algorithm \ref{alg:lp:practical} with subproblem \eqref{equ:prob:qap:lp:epsilon}  replaced by \eqref{equ:prob:qap:lp:epsilon:cut:2},  wherein $\Dcaln$ being replaced by $\Acal$, and obtain $X_K^{\best}$.\\
 Set $\Acal =  \Acal \cap \mbox{LC1}(X_K^{\best}) \cap \mbox{LC2}(X_K^{\best})$.\\
Set $\mu= \mu/2,$ $X^{\best}= \argmin_{K} \{f(X_K^{\best})\},$ and $K= K + 1$.
}
   \caption{An $L_p$-CP-negProx algorithm for problem \eqref{equ:prob:optperm}. }\label{alg:lp:cp:negProx}
  \end{algorithm}

We give some remarks on Algorithms \ref{alg:lp:negProx} and \ref{alg:lp:cp:negProx} to conclude this section.

Firstly, the term $-  \big\|X - \frac{1}{K}{\sum_{i = 1}^K \hat X_i}\big\|_{\Fsf}^2$ plays the role in pushing Algorithms \ref{alg:lp:negProx} and \ref{alg:lp:cp:negProx} to find a new permutation matrix which is far away from the average $\frac{1}{K} \sum_{i=1}^K \hat X_i$ of $\{\hat X_i\}_{i=1}^K$.  More generally, given nonnegative $\{w_i\}_{i=1}^K$ satisfying $\sum_{i=1}^K w_i = 1,$ we can choose $- \| X - \sum_{i = 1}^K w_i \hat X_i\|_{\Fsf}^2.$

Secondly, Theorem \ref{theorem:lp:epsilon:exact:negProx} shows that problem
\eqref{equ:prob:qap:lp:epsilon:cut:2} with any $\mu \in (0, \frac{c_2}{2 n})$
shares the same global solution with problem \eqref{equ:prob:optperm}. However,
(i) since $c_2$ is generally unknown, it is hard to choose the parameter $\mu$
satisfying $\mu<\frac{c_2}{2 n}$ and (ii) $\mu \in (0, \frac{c_2}{2 n})$ might be
too small to effectively improve the performance of Algorithms
\ref{alg:lp:negProx} and \ref{alg:lp:cp:negProx}. Therefore, we initialize $\mu$
with a relatively large value and gradually decrease it by setting $\mu =
\mu/2$ in the two algorithms. We terminate the two algorithms by checking that
if the newly obtained permutation matrix $\hat X_K$ has already been explored or
not; see line 2 of Algorithms  \ref{alg:lp:negProx} and \ref{alg:lp:cp:negProx}.

\section{Numerical results on QAPLIB}
In this section, we report numerical results to demonstrate the efficiency and effectiveness of our proposed algorithms for solving QAP \eqref{equ:prob:qap}, which is a special and important case of problem \eqref{equ:prob:optperm}.

We consider 134 instances from QAPLIB \cite{burkard1997qaplib} except
\verb|esc16f| and \verb|tai10b|  since the elements of matrix $A$ of \verb|esc16f|
are all zero and the best feasible solution of \verb|tai10b| is not provided.
All the experiments  were performed in OS X 10.10 on  an iMac with a 3.2GHz Intel Core i5 Processor  with access to 8GB of RAM.  We implemented our methods in MATLAB (Release 2014b).

For each problem instance, we scale the matrix $A$ and $B$ as   $A \coloneqq  A/\rho_A$ and $B \coloneqq B/\rho_B$ with $ \rho_A =  \max_{ij} |A_{ij}|$ and $\rho_B = \max_{ij} |B_{ij}|$.
We use the relative gap ratio
\[
\mathrm{gap} \coloneqq \left(\frac{\mathrm{obj} -  \mathrm{obj}^*}{\mathrm{obj}^*} \times 100\right)\%
\] to measure the quality of the solution returned by different algorithms, where $\mathrm{obj}^*$ is the optimal or best known function value provided by QAPLIB and $\mathrm{obj}$ is the function value obtained by each algorithm.  Note that the values of   $\mathrm{obj}^*$  and obj are computed based on the original $A$ and $B$, and all $\mathrm{obj}^*$ are positive integers.

\subsection{Implementation details}\label{subsection:impDetails}

For Algorithm \ref{alg:lp:practical}, we choose the initial point $X_0 =
\frac{1}{n} \1$, the initial guess $\epsilon_0 = 0.1$, the safeguards
$\epsilon_{\min} = 10^{-3}$, $\sigma_{\max} = 10^6$, and the shrinkage
parameter for updating $\epsilon_k$ as $\gamma = 0.9$. Our experience
shows that Algorithm \ref{alg:lp:practical} also works on an $\epsilon_0$ larger
than $1$, smaller shrinkage and safeguard parameters,   but the choices of these parameter might affect its efficiency.  We set the stopping tolerance $\mathrm{tol} = 10^{-3}$, which guarantees that the returned $X_k$ is sufficiently close to some permutation matrix. 
For the  projected gradient method,  we choose the initial stepsize $\alpha_0 = 10^{-3}$, the parameters of nonmonotone  linesearch condition \eqref{equ:nmls} as
$\theta = 10^{-4}$, $\delta = 0.5$, $\eta = 0.85$,   the initial guess of the tolerance for solving the subproblems \eqref{equ:prob:qap:lp:epsilon} as $\tau_0^x = 10^{-3},
\tau_0^f = 10^{-6}$ and the corresponding safeguards as $\tau_{\min}^x = 10^{-5}, \tau_{\min}^f = 10^{-8}$. For
problems with $n > 50$, we use MOSEK 7.1\footnote{Downloadable from \url{https://www.mosek.com/resources/downloads}. } to compute the projection onto $\Dcaln$ if $\max_{ij} |X_{ij}| > 8$, and use dualBB otherwise. For problems with $n \leq 50$, we always use MOSEK 7.1 to compute the projection onto $\Dcaln$.
For Algorithms \ref{alg:lp:cp} -- \ref{alg:lp:cp:negProx}, we choose $K_{\max} = 10$. As for computing projections onto the restricted domain $\mathcal{A}$ in Algorithms  \ref{alg:lp:cp}  and \ref{alg:lp:cp:negProx}, MOSEK 7.1 is always employed no matter how large $n$ is. For Algorithms \ref{alg:lp:negProx} and \ref{alg:lp:cp:negProx}, we choose $\mu = \min\left\{0.5,(\overline{\nu}_f - \underline{\nu}_f)/100\right\}$, where $\overline{\nu}_f$ and $\underline{\nu}_f$  are taken as the largest and smallest eigenvalues of $B^{\Tsf} \otimes A^{\Tsf} + B \otimes A \in \Rbb^{n^2 \times n^2}$, respectively.  The parameter $\omega$ in \eqref{equ:quadcut}, which is used in Algorithms  \ref{alg:lp:cp}  and \ref{alg:lp:cp:negProx},  is chosen to be $1 - 0.5 \underline{\nu}_f$.

For a given $X \in \Dcaln$,  a  locally 2-optimal permutation matrix can be
generated by a procedure with two steps: (i) generating a permutation $\hat
X^{(0)} \in \Pi_n$ by some greedy procedure; (ii) finding a permutation matrix
$\hat X^{(k)} \in \Ncal_2^{\best}(\hat X^{(k)})$ starting from $\hat X^{(0)}$.
Their details are outlined as follows.

Step (i). Given $X \in \Dcaln$, we use a fast greedy procedure in the code ``LagSA'' of
\cite{xia2010efficient} to generate $\hat X^{(0)} \in \Pi_n$. It computes
an approximate solution of the linear assignment problem $\max_{Y \in
\Pi_n} \langle X, Y \rangle$ in a greedy way as follows:  (a) Calculate a vector
$x \in \Rbb^n$ with the $i$-th element $x_i= \max_{j \in \Ncal} X_{ij}$. (b)
Identity a permutation vector $\pi \in \Rbb^n$ such that $x_{\pi_1} \leq
x_{\pi_2} \leq \cdots \leq x_{\pi_n}$.  (c) Let $\hat X^{(0)} \coloneqq \0$.
For each $i  \in \Ncal$, compute $k_{\pi_i} = \arg\,\max_{j \in \Ncal\backslash
\{k_{\pi_1}, \ldots, k_{\pi_{i-1}}\}} X_{\pi_i, j}$ with $\{k_{\pi_0}\} =
\emptyset$ and set  $\hat X^{(0)}_{\pi_i, k_{\pi_i}} = 1$.  We can also
apply the canonical Hungarian algorithm to solve the linear assignment problem.
The complexity  of the Hungarian algorithm and ``LagSA'' is
$O(n^3)$ and  $O(n^2)$, respectively.

 Step (ii). Starting from  $\hat X^{(0)}$, we perform an iterative local
 2-neighborhood search to find a locally 2-optimal permutation matrix. More
 specifically, for $k \geq 0$, we compute
     $\hat X^{(k)}_s \in \arg \min\limits_{Y \in \Ncal_2(\hat X^{(k)})} f(Y)$
     and update
     \be \label{equ:ls:Y:1}
     \hat X^{(k+1)} =  \begin{cases}  \hat X^{(k)}_s, &  \text{if}~f(\hat X^{(k)}_s) < f(\hat X^{(k)}), \\   \hat X^{(k)}, & \text{otherwise}.
     \end{cases}
     \ee
     The above procedure is terminated until a locally 2-optimal permutation matrix is found, i.e., $\hat X^{(k+1)} = \hat X^{(k)}$ for some $k.$

The above iterative local 2-neighborhood search procedure terminates within
$\lceil (f(\hat X^{(0)}) - f^*)/c_2 \rceil$ iterations. This is because of the
facts $f(\hat X^{(k)}) \leq f(\hat X^{(k-1)}) - c_2 \leq f(\hat X^{(0)}) - kc_2$
(i.e., \eqref{equ:fxfy}) and $f(\hat X^{(k)}) \geq f^*.$ The cost of computing
$\hat X^{(k)}_s$ is cheap for QAP. Suppose the procedure terminates within $k_0$
iterations. Then the total computational cost is $4n^3 + {\cal O}(k_0n^2)$; see
\cite{misevicius2012implementation, taillard1991robust, taillard1995comparison}
for detailed information.

  \subsection{Comparisons of $L_p$ and $L_2$ regularization algorithms}\label{subsection:lp:comp:0}

In this subsection, we compare our proposed $L_p$ regularization algorithms
(Algorithm \ref{alg:lp:practical}) using $p = 0.25, 0.5, 0.75$ with the  $L_2$
regularization algorithm LagSA \cite{xia2010efficient}. We also test a variant
of LagSA, where subproblem \eqref{equ:prob:qap:l2} is solved by the projected
gradient method, and name the variant as LagSA-BB.
Note that LagSA and LagSA-BB do not use the local 2-neighborhood search
or the rounding techniques\footnote{The local 2-neighborhood search and the rounding techniques can be used to
improve the performance of the two algorithms.}. Therefore, for fair of
comparisons, we use a basic version of our proposed algorithm, named as
$L_p$-bs, which does not use the local 2-neighborhood search technique neither.

We summarize the numerical results in Table \ref{table:gap:level:02:14}, where
the number denotes how many problem instances can be solved by the corresponding
algorithm within the given gap. For instance, Table \ref{table:gap:level:02:14}
shows that LagSA can solve $14$ instances within gap $0.0$ (to global
optimality) and $99$ instances within gap $5.0$. It can be seen from Table
\ref{table:gap:level:02:14} that our proposed $L_p$ regularization algorithms
generally perform better than LagSA and LagSA-BB and $L_{0.75}$-bs performs the
best (among all the five algorithms). It can also be seen from Table
\ref{table:gap:level:02:14} that LagSA-BB performs better than LagSA by using
the projected gradient method to solve $L_2$ regularization subproblem \eqref{equ:prob:qap:l2} instead of the (truncated) Frank-Wolfe method.

 \begin{table}[!htbp]
\centering
\linespread{1.05}
  \begin{footnotesize}
  \caption{Comparison of gap levels of the five algorithms on 134 instances in QAPLIB} \label{table:gap:level:02:14}
  \begin{tabular}{@{}r@{\hspace{0.25cm}}r@{\hspace{0.25cm}}r@{\hspace{0.25cm}}r@{\hspace{0.25cm}}r@{\hspace{0.25cm}}r@{\hspace{0.25cm}}r@{\hspace{0.25cm}}r@{\hspace{0.25cm}}r@{\hspace{0.25cm}}r@{\hspace{0.25cm}}r@{\hspace{0.25cm}}r@{\hspace{0.25cm}}r@{\hspace{0.25cm}}r@{\hspace{0.25cm}}r@{\hspace{0.25cm}}r@{}}
  \hline
\Gape[4pt] gap $\leq $ \% & 0.0  & 0.1  & 0.2  & 0.3  & 0.4  & 0.5  & 0.6  & 0.7  & 0.8  & 0.9  & 1.0    & 2.0   & 3.0   & 4.0    & 5.0 \\
 \Xcline{1-16}{.8pt}
        LagSA &  14 & 21 & 29 & 41 & 42 & 44 & 48 & 51 & 53 & 57 & 61 & 80 & 91 & 95 & 99 \\
  LagSA-BB &  26 & 39 & 50 & 56 & 58 & 62 & 65 & 68 & 73 & 78 & 80 & 96 & 105 & 111 & 116 \\
$L_{0.25}$-bs &  22 & 37 & 46 & 57 & 58 & 59 & 62 & 68 & 71 & 74 & 75 & 90 & 102 & 105 & 109 \\
$L_{0.5}$-bs &  24 & 37 & 50 & 56 & 59 & 60 & 64 & 70 & 76 & 77 & 77 & 89 & 107 & 111 & 113 \\
$L_{0.75}$-bs &  \textbf{27} & \textbf{44} & \textbf{52} & \textbf{65} & \textbf{67} & \textbf{69} & \textbf{73} &
\textbf{74} & \textbf{80} & \textbf{82} & \textbf{84} & \textbf{98} & \textbf{108} & \textbf{114} & \textbf{115} \\
                      \hline
\end{tabular}
\end{footnotesize}
\end{table}

To make the comparison clearer, we make a pairwise comparison of the five
algorithms as used in \cite{dai2005projected}.
 Given a problem and a collection of the five algorithms, algorithm $i$ is called the winner if the gap
 obtained by the algorithm is the smallest among all algorithms. We summarize the comparison results in Table
\ref{table:winners:02:14}.  The second line of Table
\ref{table:winners:02:14} reports the winners of each algorithm among all the five algorithms for solving the 134 instances. In particular, Table
\ref{table:winners:02:14} shows that $L_{0.75}$-bs wins $62$ of the 134 instances.  Starting from the third line, we perform pairwise comparisons of the algorithms. It follows from the last line of this table that $L_{0.75}$-bs shows its superiority over the other algorithms. In summary,
 $L_{0.75}$-bs performs the best among all
the five algorithms. Therefore, we set $p=0.75$ in our proposed $L_p$
regularization algorithms in the subsequent numerical experiments.
It should be pointed out that the performance of  $L_{0.75}$ is much
better than that of $L_{0.75}$-bs due to the local 2-neighborhood search and
rounding techniques. For instance,  $L_{0.75}$ can solve 51 problems to a zero
gap while $L_{0.75}$-bs  can only solve 27 problems to a zero gap.

\begin{table}[!htbp]
\centering
\linespread{1.05}
\caption{Pairwise comparison of the five algorithms}\label{table:winners:02:14}
\begin{footnotesize}
\begin{tabular}{@{}c@{:}cccccc@{}}
\hline
\multicolumn{2}{c}{methods} & LagSA & LagSA-BB & $L_{0.25}$-bs & $L_{0.5}$-bs & $L_{0.75}$-bs \\
\Xcline{1-7}{0.8pt}
\multicolumn{2}{c}{\# of winners}& 28& 53& 47& 49& \textbf{62}\\
    LagSA &   i& --- &  \ 47 : 107& 53 : 92& 49 : 98&   \,\ 46\ : \textbf{101} \\
  LagSA-BB\, &   i& 107 : 47& --- & 88 : 66& 79 : 76& 71 : \textbf{88} \\
$L_{0.25}$-bs &   i&  \ 92 \,: 53& 66 : 88& --- & 85 : 96& 69 : \textbf{99} \\
$L_{0.5}$-bs &   i&  \ 98 \,: 49& 76 : 79& 96 : 85& --- & 86 : \textbf{93} \\
$L_{0.75}$-bs &   i& \textbf{101}\,: 46& \textbf{88}\,: 71& \textbf{99}\,: 69& \textbf{93}\,: 86& ---  \\
\hline
\end{tabular}
\end{footnotesize}
\end{table}

\subsection{Comparisons of $L_{0.75}$-negProx, $L_{0.75}$-CP, and $L_{0.75}$-CP-negProx}\label{subsection:negProx}

Numerical results of $L_{0.75}$-negProx, $L_{0.75}$-CP, and
$L_{0.75}$-CP-negProx are presented in Tables \ref{table:lp:negProx:less80} and
\ref{table:lp:negProx:greater80}.  For the sake of saving space, we
only present in Table  \ref{table:lp:negProx:less80} the results on problem
instances such that the best gap returned by $L_{0.75}$-negProx, $L_{0.75}$-CP, and $L_{0.75}$-CP-negProx is zero.

In the tables in this and subsequent subsections,  ``Rgap'' denotes the average
relative gap of 100,000 random permutation matrices, which can illustrate the
hardness of the corresponding problem to some extent; ``gap'' denotes the
relative gap achieved by our proposed algorithm; and ``time'' denotes the
running time in seconds of different algorithms. We put the best gaps
among all the gaps returned by different algorithms in bold. In the last line of
each table, we summarize the total number of instances $N_{\best}$ such that the gap returned by each algorithm is the best among the gaps returned by all algorithms.

It can be seen from Tables \ref{table:lp:negProx:less80} and \ref{table:lp:negProx:greater80} that the performance of $L_{0.75}$-negProx and $L_{0.75}$-CP are better than that of $L_{0.75}$ on most of problem instances. This shows that both  the negative proximal point technique and the cutting plane technique are helpful in improving the performance of $L_{0.75}.$ It can also be seen from Tables \ref{table:lp:negProx:less80} and \ref{table:lp:negProx:greater80} that $L_{0.75}$-negProx (which employs the negative proximal point technique) performs much better than $L_{0.75}$-CP (which uses the cutting plane technique). In particular, $L_{0.75}$-negProx reduces the gaps of 17 problem instances to zero and $L_{0.75}$-CP reduces the gaps of only 5 problem instances to zero. This demonstrates that our proposed negative proximal point technique is more effective in improving the performance of $L_{0.75}$.

It can also be observed from the two tables that $L_{0.75}$-CP-negProx (which
employs both the cutting plane technique and the negative proximal point
technique) performs slightly better in terms of the solution quality than
$L_{0.75}$-negProx.  However, $L_{0.75}$-negProx is faster than
$L_{0.75}$-CP-negProx since the projections in $L_{0.75}$-CP-negProx are more
difficult to compute than those in $L_{0.75}$-negProx.

\begin{table}[!htbp]
 \centering
 \linespread{1.0}
 \caption{Numerical results of  $L_{0.75}$-negProx, $L_{0.75}$-CP, $L_{0.75}$-CP-negProx  for 19 problem instances with $n < 80$ in QAPLIB } \label{table:lp:negProx:less80}
\begin{footnotesize}
  \begin{tabular}{@{}l @{\hspace{1.5mm}}r @{\hspace{0mm}}c @{\hspace{6.0mm}}   r @{\hspace{2.5mm}}  r @{\hspace{0mm}} c @{\hspace{6.0mm}} r @{\hspace{2.5mm}}  r @{\hspace{0mm}} c @{\hspace{6.0mm}}  r @{\hspace{2.5mm}}  r @{\hspace{0mm}} c@{\hspace{6.0mm}} r@{\hspace{2.5mm}}r@{}}
\hline
 Problem   &  &&   \multicolumn{2}{c}{$L_{0.75}$}  &&  \multicolumn{2}{c}{$L_{0.75}$-}  &&  \multicolumn{2}{c}{$L_{0.75}$-} &&  \multicolumn{2}{c}{$L_{0.75}$-} \\
  &  &&   \multicolumn{2}{c}{}  &&  \multicolumn{2}{c}{negProx}  &&  \multicolumn{2}{c}{CP} &&  \multicolumn{2}{c}{CP-negProx} \\
\cline{1-2} \cline{4-5}    \cline{7-8} \cline{10-11}  \cline{13-14}
\Gape[4pt] name &  Rgap   &&   gap&  time &&  gap&  time&&  gap&  time &  &   gap  & time  \\
 \Xhline{0.8pt}
    bur26a  & 9.6 &&0.1212 & 0.8  &&\textbf{0.0000} & 2.6  &&0.0889 & 3.3  &&\textbf{0.0000} & 3.0   \\ 
    bur26b  &10.4 &&0.1769 & 1.0  &&\textbf{0.0000} & 3.0  &&0.1716 & 4.6  &&\textbf{0.0000} & 3.1   \\ 
    bur26d  &10.2 &&0.0020 & 0.6  &&\textbf{0.0000} & 3.6  &&0.0020 & 3.3  &&\textbf{0.0000} & 4.3   \\ 
    bur26e  &10.3 &&0.0091 & 0.7  &&\textbf{0.0000} & 2.2  &&\textbf{0.0000} & 4.0  &&\textbf{0.0000} & 2.5   \\ 
    chr15c  &546.6 &&30.9554 & 0.4  &&20.1178 & 3.9  &&30.9554 & 2.0  &&\textbf{0.0000} & 3.6   \\ 
    chr20c  &656.5 &&18.2152 & 0.4  &&\textbf{0.0000} & 3.5  &&18.2152 & 1.7  &&\textbf{0.0000} & 4.3   \\ 
    esc32b  &171.0 &&9.5238 & 1.8  &&\textbf{0.0000} &10.5  &&\textbf{0.0000} &19.9  &&\textbf{0.0000} &10.6   \\ 
     had20  &12.2 &&0.0289 & 0.5  &&\textbf{0.0000} & 2.2  &&0.0867 & 2.5  &&\textbf{0.0000} & 2.6   \\ 
    kra30a  &51.5 &&0.8999 & 1.1  &&1.3498 &10.2  &&1.5298 &12.4  &&\textbf{0.0000} &10.8   \\ 
   lipa20a  & 7.0 &&2.0092 & 0.1  &&\textbf{0.0000} & 2.3  &&1.8192 & 0.6  &&\textbf{0.0000} & 4.6   \\ 
     nug14  &34.5 &&0.1972 & 0.3  &&\textbf{0.0000} & 1.7  &&0.1972 & 1.3  &&\textbf{0.0000} & 1.9   \\ 
     nug18  &32.9 &&0.4145 & 0.5  &&\textbf{0.0000} & 4.3  &&0.4145 & 2.0  &&\textbf{0.0000} & 5.6   \\ 
     nug28  &34.2 &&0.1549 & 1.3  &&\textbf{0.0000} & 5.5  &&\textbf{0.0000} & 6.0  &&\textbf{0.0000} & 6.7   \\ 
     scr12  &89.2 &&2.3368 & 0.1  &&\textbf{0.0000} & 1.0  &&\textbf{0.0000} & 0.9  &&\textbf{0.0000} & 1.4   \\ 
     scr20  &105.7 &&0.6525 & 0.3  &&\textbf{0.0000} & 4.0  &&0.0254 & 2.4  &&\textbf{0.0000} & 2.5   \\ 
    ste36b  &439.0 &&4.5672 & 0.7  &&\textbf{0.0000} & 4.0  &&1.1734 & 3.1  &&1.5519 & 8.3   \\ 
    tai20b  &163.6 &&0.4526 & 0.3  &&\textbf{0.0000} & 2.3  &&0.4526 & 1.4  &&\textbf{0.0000} & 2.8   \\ 
    tai30a  &21.0 &&1.7887 & 2.2  &&\textbf{0.0000} & 7.7  &&1.7887 &10.6  &&\textbf{0.0000} & 9.0   \\ 
    tai64c  &59.2 &&0.0926 & 5.3  &&\textbf{0.0000} &24.2  &&\textbf{0.0000} &12.4  &&\textbf{0.0000} &21.1   \\ 
\revise{$N_{\best}$}&&&\multicolumn{1}{c}{\revise{ 0}} && &\multicolumn{1}{c}{\revise{17}} && &\multicolumn{1}{c}{\revise{5}} && &\multicolumn{1}{c}{\revise{18}} & \\ 
\hline
     \end{tabular}
\end{footnotesize}
\end{table}

\begin{table}[!htbp]
 \centering
 \linespread{1.0}
   \caption{Numerical results of  $L_{0.75}$-negProx, $L_{0.75}$-CP, $L_{0.75}$-CP-negProx  for 14 problem instances with $n \geq 80$ in QAPLIB } \label{table:lp:negProx:greater80}
\begin{footnotesize}
  \begin{tabular}{@{}l @{\hspace{1.5mm}}r @{\hspace{0mm}}c @{\hspace{6.0mm}}   r @{\hspace{2.5mm}}  r @{\hspace{0mm}} c @{\hspace{6.0mm}} r @{\hspace{2.5mm}}  r @{\hspace{0mm}} c @{\hspace{6.0mm}}  r @{\hspace{2.5mm}}  r @{\hspace{0mm}} c@{\hspace{6.0mm}} r@{\hspace{2.5mm}}r@{}}
\hline
 Problem   &  &&   \multicolumn{2}{c}{$L_{0.75}$}  &&  \multicolumn{2}{c}{$L_{0.75}$-}  &&  \multicolumn{2}{c}{$L_{0.75}$-} &&  \multicolumn{2}{c}{$L_{0.75}$-} \\
  &  &&   \multicolumn{2}{c}{}  &&  \multicolumn{2}{c}{negProx}  &&  \multicolumn{2}{c}{CP} &&  \multicolumn{2}{c}{CP-negProx} \\
\cline{1-2} \cline{4-5}    \cline{7-8} \cline{10-11}  \cline{13-14}
\Gape[4pt] name &  Rgap   &&   gap&  time &&  gap&  time&&  gap&  time &  &   gap  & time  \\
 \Xhline{0.8pt}
   lipa80a  & 2.1 &&0.7540 & 0.4  &&0.7255 & 3.6  &&\textbf{0.7168} & 2.8  &&0.7255 & 6.3   \\ 
   lipa90a  & 1.9 &&0.7179 & 0.9  &&\textbf{0.6827} & 7.7  &&0.6960 & 3.2  &&0.6893 & 4.3   \\ 
   sko100b  &17.4 &&0.0858 &13.9  &&0.0832 &141.4  &&\textbf{0.0793} &177.9  &&0.0858 &127.8   \\ 
   sko100d  &17.5 &&0.1712 &13.9  &&\textbf{0.1043} &97.0  &&0.1712 &103.9  &&\textbf{0.1043} &157.7   \\ 
   sko100e  &18.2 &&0.0134 &14.0  &&\textbf{0.0000} &88.2  &&0.0094 &95.7  &&\textbf{0.0000} &117.7   \\ 
   sko100f  &17.3 &&0.0550 &12.3  &&\textbf{0.0523} &199.9  &&0.0550 &95.7  &&0.0550 &92.6   \\ 
     sko81  &19.2 &&0.1143 &12.2  &&\textbf{0.0813} &71.0  &&0.1143 &50.4  &&\textbf{0.0813} &94.1   \\ 
   tai100b  &50.9 &&0.3800 &13.7  &&\textbf{0.2237} &62.8  &&0.3789 &76.9  &&\textbf{0.2237} &90.7   \\ 
   tai150b  &30.6 &&0.5098 &86.2  &&\textbf{0.1422} &221.5  &&0.4995 &836.6  &&0.1462 &330.2   \\ 
   tai256c  &19.8 &&0.2610 &24.6  &&0.1112 &311.0  &&\textbf{0.1041} &493.2  &&0.1581 &277.3   \\ 
    tai80a  &15.6 &&0.6904 &20.0  &&\textbf{0.4734} &111.8  &&0.6904 &93.6  &&\textbf{0.4734} &129.0   \\ 
    tai80b  &51.9 &&0.0378 & 6.7  &&0.0378 &36.0  &&\textbf{0.0336} &53.4  &&0.0378 &53.6   \\ 
    tho150  &20.6 &&0.1500 &53.4  &&0.1302 &225.7  &&0.1302 &377.5  &&\textbf{0.1223} &478.2   \\ 
    wil100  & 9.8 &&0.0146 &21.0  &&0.0146 &91.7  &&\textbf{0.0132} &177.9  &&\textbf{0.0132} &277.1   \\ 
\revise{$N_{\best}$}&&&\multicolumn{1}{c}{\revise{ 0}} && &\multicolumn{1}{c}{\revise{ 8}} && &\multicolumn{1}{c}{\revise{ 5}} && &\multicolumn{1}{c}{\revise{ 7}} & \\ 
\hline
     \end{tabular}
\end{footnotesize}
\end{table}

\subsection{Comparisons of $L_{0.75}$, Ro-TS, and their hybrid $L_{0.75}$-Ro-TS} \label{subsection:lp:comp}
 To illustrate the efficiency and effectiveness of our proposed $L_{0.75}$
 regularization algorithm (Algorithm \ref{alg:lp:practical} with $p=0.75$), we
 compare it with one of the state-of-the-art heuristics
 Ro-TS\footnote{Downloadable from
 \url{http://mistic.heig-vd.ch/taillard/codes.dir/tabou_qap2.c}} in this
 subsection. Ro-TS implements the robust taboo search for solving QAP. For QAP
 problem instances with $n<80$, we set the parameters in Ro-TS to be
 \verb|nr_iterations = 1500*n; nr_resolutions = 10|; for instances with $n \geq
 80$, we set \verb|nr_iterations = 500*n; nr_resolutions = 10.| We also test a
 hybrid, named as $L_{0.75}$-Ro-TS, of $L_{0.75}$ and Ro-TS by using the
 permutation matrix returned by $L_{0.75}$ as the input of Ro-TS. The
 combination of $L_{0.75}$-negProx, $L_{0.75}$-CP, and $L_{0.75}$-CP-negProx
 with Ro-TS can improve the performance similarly but their results are not
 shown for the simplicity of presentation.

In the tables in this subsection, the numbers ``min gap'', ``mean gap'', and
``max gap'' denote the minimum, mean, and maximum gap of Ro-TS among 10 runs,
respectively; ``nfe'' denotes the total number of the objective function
evaluations of  \eqref{equ:prob:qap:lp:epsilon} with $p={0.75}$;
`$\overline{\text{time}}$'' denotes the mean running time in seconds of
different algorithms. Note that the implementation of Ro-TS is in the C
language, while our proposed algorithm is implemented in the MATLAB environment.
Hence, the comparison of the running time is more favorable for Ro-TS.

In Table \ref{table:lp:tabu:less80}, we present numerical results on 51 problem instances in QAPLIB for which our proposed algorithm $L_{0.75}$ is able to achieve the optimal value or the best  known upper bound.
 It can be observed from Table \ref{table:lp:tabu:less80} that the time used by
 our algorithm is comparable to that of Ro-TS.
 Particularly, for the problem instance ``esc128'',
 our algorithm is twice faster than Ro-TS.

\begin{center}
 \linespread{1.0}
 \begin{footnotesize}
  \begin{longtable}[!htbp]{@{}l @{\hspace{3.0mm}} r @{\hspace{3.0mm}} r @{\hspace{0mm}}c @{\hspace{5.0mm}} r @{\hspace{3.0mm}} r @{\hspace{3.0mm}}  r @{\hspace{0mm}} c@{\hspace{5.0mm}} r@{\hspace{3.0mm}}r@{}}
 \caption{Numerical results of $L_{0.75}$ and Ro-TS on 51 problem instances in QAPLIB}\label{table:lp:tabu:less80}
   ~\\[0.6pt]
\hline
\Gape[4pt] Problem &   &  &&   \multicolumn{3}{c}{$L_{0.75}$} &&  \multicolumn{2}{c}{Ro-TS, 10 runs} \\
\cline{1-3} \cline{5-7}    \cline{9-10}
\Gape[5pt] name & $\mathrm{obj}^*$ &  Rgap   &&  gap&  time &nfe &  &   gap  & $\overline{\text{time}}$  \\
 & &     &&  &     & &  &    (min, mean, max) &  \\
   \Xcline{1-10}{0.8pt}
 \endfirsthead
 \caption{continued from the previous page}
 \\[6pt]
\hline
\Gape[4pt] Problem &   &  &&   \multicolumn{3}{c}{$L_{0.75}$} &&  \multicolumn{2}{c}{Ro-TS, 10 runs} \\
\cline{1-3} \cline{5-7}    \cline{9-10}
\Gape[5pt] name & $\mathrm{obj}^*$ &  Rgap   &&  gap&  time &nfe &  &   gap  &$\overline{\text{time}}$ \\
 & &     &&  &     & &  &    (min, mean, max) &   \\
  \Xcline{1-10}{0.8pt}
 \endhead
 \multicolumn{10}{r}{Continued on next page}
 \\ \hline
\endfoot
\endlastfoot
    bur26c &    5426795 & 9.5 &&\textbf{0.0000} & 0.7  & 161 && (\textbf{0.0000}, 0.0004, 0.0036) &  1.0\\ 
    bur26f &    3782044 &11.3 &&\textbf{0.0000} & 1.2  & 246 && (\textbf{0.0000}, 0.0001, 0.0006) &  1.0\\ 
    bur26g &   10117172 & 9.9 &&\textbf{0.0000} & 0.6  & 147 && (\textbf{0.0000}, 0.0000, 0.0000) &  1.0\\ 
    bur26h &    7098658 &10.9 &&\textbf{0.0000} & 0.6  & 149 && (\textbf{0.0000}, 0.0003, 0.0035) &  1.0\\ 
    chr12a &       9552 &372.5 &&\textbf{0.0000} & 0.3  & 185 && (\textbf{0.0000}, 0.0000, 0.0000) &  0.1\\ 
    chr12b &       9742 &363.4 &&\textbf{0.0000} & 0.2  & 152 && (\textbf{0.0000}, 0.0000, 0.0000) &  0.1\\ 
    chr18b &       1534 &199.9 &&\textbf{0.0000} & 1.0  & 366 && (\textbf{0.0000}, 0.0000, 0.0000) &  0.3\\ 
    esc16a &         68 &63.3 &&\textbf{0.0000} & 0.4  & 199 && (\textbf{0.0000}, 0.0000, 0.0000) &  0.2\\ 
    esc16b &        292 & 7.9 &&\textbf{0.0000} & 3.8  & 269 && (\textbf{0.0000}, 0.0000, 0.0000) &  0.2\\ 
    esc16c &        160 &55.8 &&\textbf{0.0000} & 0.4  & 195 && (\textbf{0.0000}, 0.0000, 0.0000) &  0.2\\ 
    esc16d &         16 &226.0 &&\textbf{0.0000} & 0.7  & 308 && (\textbf{0.0000}, 0.0000, 0.0000) &  0.2\\ 
    esc16e &         28 &118.5 &&\textbf{0.0000} & 0.3  & 158 && (\textbf{0.0000}, 0.0000, 0.0000) &  0.2\\ 
    esc16g &         26 &152.6 &&\textbf{0.0000} & 0.3  & 120 && (\textbf{0.0000}, 0.0000, 0.0000) &  0.2\\ 
    esc16h &        996 &41.6 &&\textbf{0.0000} & 0.4  & 187 && (\textbf{0.0000}, 0.0000, 0.0000) &  0.2\\ 
    esc16i &         14 &288.7 &&\textbf{0.0000} & 0.3  & 132 && (\textbf{0.0000}, 0.0000, 0.0000) &  0.2\\ 
    esc16j &          8 &268.0 &&\textbf{0.0000} & 0.3  & 160 && (\textbf{0.0000}, 0.0000, 0.0000) &  0.2\\ 
    esc32c &        642 &45.3 &&\textbf{0.0000} & 1.0  & 227 && (\textbf{0.0000}, 0.0000, 0.0000) &  1.8\\ 
    esc32d &        200 &80.2 &&\textbf{0.0000} & 1.0  & 237 && (\textbf{0.0000}, 0.0000, 0.0000) &  1.8\\ 
    esc32e &          2 &2428.2 &&\textbf{0.0000} & 0.8  & 187 && (\textbf{0.0000}, 0.0000, 0.0000) &  1.8\\ 
    esc32g &          6 &637.6 &&\textbf{0.0000} & 0.7  & 151 && (\textbf{0.0000}, 0.0000, 0.0000) &  1.8\\ 
    esc32h &        438 &53.7 &&\textbf{0.0000} & 1.1  & 242 && (\textbf{0.0000}, 0.0000, 0.0000) &  1.8\\ 
    esc64a &        116 &140.1 &&\textbf{0.0000} & 5.7  & 443 && (\textbf{0.0000}, 0.0000, 0.0000) & 14.9\\ 
     had12 &       1652 &14.3 &&\textbf{0.0000} & 0.3  & 181 && (\textbf{0.0000}, 0.0000, 0.0000) &  0.1\\ 
     had14 &       2724 &15.7 &&\textbf{0.0000} & 0.2  & 128 && (\textbf{0.0000}, 0.0000, 0.0000) &  0.1\\ 
     had16 &       3720 &13.6 &&\textbf{0.0000} & 0.4  & 171 && (\textbf{0.0000}, 0.0000, 0.0000) &  0.2\\ 
     had18 &       5358 &11.8 &&\textbf{0.0000} & 0.4  & 177 && (\textbf{0.0000}, 0.0000, 0.0000) &  0.3\\ 
     kra32 &      88700 &54.6 &&\textbf{0.0000} & 1.0  & 192 && (\textbf{0.0000}, 0.0000, 0.0000) &  1.8\\ 
   lipa20b &      27076 &31.3 &&\textbf{0.0000} & 0.3  & 131 && (\textbf{0.0000}, 0.0000, 0.0000) &  0.4\\ 
   lipa30b &     151426 &29.2 &&\textbf{0.0000} & 0.4  & 108 && (\textbf{0.0000}, 0.0000, 0.0000) &  1.5\\ 
   lipa40b &     476581 &30.4 &&\textbf{0.0000} & 0.5  &  83 && (\textbf{0.0000}, 0.0000, 0.0000) &  3.6\\ 
   lipa50b &    1210244 &28.9 &&\textbf{0.0000} & 0.9  &  94 && (\textbf{0.0000}, 0.0000, 0.0000) &  7.1\\ 
   lipa60b &    2520135 &29.9 &&\textbf{0.0000} & 1.6  & 113 && (\textbf{0.0000}, 0.0000, 0.0000) & 12.3\\ 
   lipa70b &    4603200 &30.1 &&\textbf{0.0000} & 1.5  &  84 && (\textbf{0.0000}, 0.0000, 0.0000) & 19.6\\ 
     nug12 &        578 &40.5 &&\textbf{0.0000} & 0.3  & 191 && (\textbf{0.0000}, 0.0000, 0.0000) &  0.1\\ 
     nug15 &       1150 &37.8 &&\textbf{0.0000} & 0.4  & 220 && (\textbf{0.0000}, 0.0000, 0.0000) &  0.2\\ 
    nug16b &       1240 &39.3 &&\textbf{0.0000} & 0.3  & 152 && (\textbf{0.0000}, 0.0000, 0.0000) &  0.2\\ 
     nug20 &       2570 &32.6 &&\textbf{0.0000} & 0.6  & 205 && (\textbf{0.0000}, 0.0000, 0.0000) &  0.4\\ 
     nug21 &       2438 &40.3 &&\textbf{0.0000} & 0.5  & 162 && (\textbf{0.0000}, 0.0000, 0.0000) &  0.5\\ 
     nug22 &       3596 &43.2 &&\textbf{0.0000} & 0.6  & 186 && (\textbf{0.0000}, 0.0000, 0.0000) &  0.6\\ 
     nug24 &       3488 &36.7 &&\textbf{0.0000} & 0.7  & 191 && (\textbf{0.0000}, 0.0000, 0.0000) &  0.8\\ 
     nug25 &       3744 &33.7 &&\textbf{0.0000} & 0.6  & 158 && (\textbf{0.0000}, 0.0000, 0.0000) &  0.9\\ 
     nug27 &       5234 &36.2 &&\textbf{0.0000} & 0.6  & 147 && (\textbf{0.0000}, 0.0000, 0.0000) &  1.1\\ 
     rou15 &     354210 &32.2 &&\textbf{0.0000} & 0.9  & 544 && (\textbf{0.0000}, 0.0000, 0.0000) &  0.2\\ 
     scr15 &      51140 &99.5 &&\textbf{0.0000} & 0.3  & 149 && (\textbf{0.0000}, 0.0000, 0.0000) &  0.2\\ 
    tai10a &     135028 &38.4 &&\textbf{0.0000} & 0.3  & 177 && (\textbf{0.0000}, 0.0000, 0.0000) &  0.1\\ 
    tai12a &     224416 &39.3 &&\textbf{0.0000} & 0.3  & 159 && (\textbf{0.0000}, 0.0000, 0.0000) &  0.1\\ 
    tai12b &   39464925 &111.8 &&\textbf{0.0000} & 0.3  & 178 && (\textbf{0.0000}, 0.0000, 0.0000) &  0.1\\ 
    tai15b &   51765268 &676.6 &&\textbf{0.0000} & 0.2  &  70 && (\textbf{0.0000}, 0.0000, 0.0000) &  0.2\\ 
    esc128 &         64 &397.6 &&\textbf{0.0000} &19.8  & 390 && (\textbf{0.0000}, 13.4375, 28.1250) & 41.5\\ 
   lipa80b &    7763962 &30.8 &&\textbf{0.0000} & 3.1  & 163 && (\textbf{0.0000}, 0.0000, 0.0000) &  9.8\\ 
   lipa90b &   12490441 &30.5 &&\textbf{0.0000} & 2.2  &  84 && (\textbf{0.0000}, 0.0000, 0.0000) & 14.1\\ 
\revise{$N_{\best}$}&& &&\multicolumn{1}{l}{\quad \revise{51}} & & &&\multicolumn{1}{l}{\quad  \revise{51}} & \\ 
\hline
     \end{longtable}
\end{footnotesize}
\end{center}

There are still 65 problem instances with $n < 80$, for  which $L_{0.75}$ cannot
achieve the optimal value or the best  known upper bound. For the sake of saving space, we only report in Table \ref{table:lp:tabu:less80:combination:2} the results on problem instances, for which the maximal gap returned by $L_{0.75}$-Ro-TS or  Ro-TS in 10 runs is greater than zero. For these problems, although $L_{0.75}$ does not perform as well as Ro-TS, the hybrid $L_{0.75}$-Ro-TS can achieve a satisfactory performance (compared with Ro-TS).
Note that the time of $L_{0.75}$-Ro-TS does not include the time of $L_{0.75}$.

\begin{center}
 \linespread{1.0}
\begin{footnotesize}
  \begin{longtable}{@{}l @{\hspace{0.2mm}}r @{\hspace{0mm}}c @{\hspace{3.0mm}}   r @{\hspace{1.2mm}}r @{\hspace{1.8mm}}  r @{\hspace{0mm}} c @{\hspace{3.4mm}}  r @{\hspace{1.2mm}}  r @{\hspace{0mm}} c@{\hspace{3.0mm}} r@{\hspace{1.2mm}}r@{}}
     \caption{Numerical results of $L_{0.75}$, $L_{0.75}$-Ro-TS, and Ro-TS on 41 problem instances with $n < 80$ in QAPLIB } \label{table:lp:tabu:less80:combination:2}
     ~\\[0.6pt]
\hline
 \Gape[4pt]Problem   &  &&   \multicolumn{3}{c}{$L_{0.75}$}  &&  \multicolumn{2}{c}{$L_{0.75}$-Ro-TS, 10 runs} &&  \multicolumn{2}{c}{Ro-TS, 10 runs} \\
\cline{1-2} \cline{4-6}    \cline{8-9} \cline{11-12}
\Gape[5pt] name &  Rgap   &&   gap&  time & nfe &&  gap&  $\overline{\text{time}}$&  &   gap  & $\overline{\text{time}}$  \\
 &      &&  &     &  &&  (min, mean, max) &  &     &  (min, mean, max) & \\
 \Xcline{1-12}{0.8pt}
 \endfirsthead
  \caption{continued from the previous page}  ~\\[0.6pt]
 \hline
 \Gape[4pt]Problem   &  &&   \multicolumn{3}{c}{$L_{0.75}$}  &&  \multicolumn{2}{c}{$L_{0.75}$-Ro-TS, 10 runs} &&  \multicolumn{2}{c}{Ro-TS, 10 runs} \\
\cline{1-2} \cline{4-6}    \cline{8-9} \cline{11-12}
\Gape[5pt] name &  Rgap   &&   gap&  time & nfe &&  gap&  $\overline{\text{time}}$&  &   gap  & $\overline{\text{time}}$  \\
 &      &&  &     &  &&  (min, mean, max) &  &     &  (min, mean, max) & \\
 \Xcline{1-12}{0.8pt}
 \endhead
  \multicolumn{12}{r}{Continued on next page} \\ \hline
\endfoot
\endlastfoot
    bur26a & 9.6 &&0.1212  & 0.8 & 188 && (\textbf{0.0000}, 0.0137, 0.0845) &  1.0  && (\textbf{0.0000}, 0.0000, 0.0000) &  1.0\\ 
    bur26b &10.4 &&0.1769  & 1.0 & 259 && (\textbf{0.0000}, 0.0022, 0.0189) &  1.0  && (\textbf{0.0000}, 0.0000, 0.0000) &  1.0\\ 
    bur26d &10.2 &&0.0020  & 0.6 & 152 && (\textbf{0.0000}, 0.0001, 0.0008) &  1.0  && (\textbf{0.0000}, 0.0006, 0.0021) &  1.0\\ 
    chr15a &520.4 &&0.4042  & 0.5 & 276 && (\textbf{0.0000}, 0.0404, 0.4042) &  0.2  && (\textbf{0.0000}, 0.0829, 0.8286) &  0.2\\ 
    chr15b &668.1 &&17.6971  & 0.3 & 136 && (\textbf{0.0000}, 0.5507, 2.7534) &  0.2  && (\textbf{0.0000}, 0.0000, 0.0000) &  0.2\\ 
    chr15c &546.6 &&30.9554  & 0.4 & 243 && (\textbf{0.0000}, 0.0000, 0.0000) &  0.2  && (\textbf{0.0000}, 0.6355, 6.3552) &  0.2\\ 
    chr18a &600.0 &&11.7679  & 0.5 & 221 && (\textbf{0.0000}, 0.0000, 0.0000) &  0.3  && (\textbf{0.0000}, 0.0396, 0.3965) &  0.3\\ 
    chr20a &388.5 &&12.5000  & 0.8 & 274 && (\textbf{0.0000}, 2.4453, 6.3869) &  0.4  && (\textbf{0.0000}, 1.8887, 5.4745) &  0.4\\ 
    chr20b &366.0 &&13.8381  & 0.7 & 271 && (3.0461, 4.4822, 5.9182) &  0.4  && (\textbf{0.0000}, 4.3777, 6.6144) &  0.4\\ 
    chr20c &656.5 &&18.2152  & 0.4 & 141 && (\textbf{0.0000}, 2.9359, 10.4370) &  0.4  && (\textbf{0.0000}, 1.8894, 4.7235) &  0.4\\ 
    chr22a &153.8 &&5.7180  & 0.5 & 163 && (\textbf{0.0000}, 0.7797, 1.8194) &  0.6  && (\textbf{0.0000}, 0.6790, 2.1767) &  0.6\\ 
    chr22b &152.3 &&5.2954  & 0.6 & 192 && (\textbf{0.0000}, 0.9106, 1.8405) &  0.6  &&(0.7104, 1.3497, 1.8082) &  0.6\\ 
    chr25a &423.8 &&25.2371  & 0.9 & 279 && (\textbf{0.0000}, 7.1286, 11.6965) &  0.9  &&(2.3182, 5.8061, 10.6428) &  0.9\\ 
    esc32a &233.2 &&1.5385  & 3.1 & 440 && (\textbf{0.0000}, 0.6154, 1.5385) &  1.8  && (\textbf{0.0000}, 0.0000, 0.0000) &  1.8\\ 
    kra30a &51.5 &&0.8999  & 1.1 & 222 && (\textbf{0.0000}, 0.0000, 0.0000) &  1.5  && (\textbf{0.0000}, 0.3690, 1.3386) &  1.5\\ 
    kra30b &49.9 &&0.3172  & 1.3 & 274 && (\textbf{0.0000}, 0.0077, 0.0766) &  1.5  && (\textbf{0.0000}, 0.0077, 0.0766) &  1.5\\ 
   lipa60a & 2.7 &&0.9858  & 0.3 & 38 && (\textbf{0.0000}, 0.0671, 0.6706) & 12.3  && (\textbf{0.0000}, 0.0000, 0.0000) & 12.3\\ 
   lipa70a & 2.4 &&0.8807  & 0.3 & 36 && (\textbf{0.0000}, 0.0596, 0.5956) & 19.7  && (\textbf{0.0000}, 0.0590, 0.5903) & 19.6\\ 
     nug30 &32.8 &&0.0653  & 1.0 & 205 && (\textbf{0.0000}, 0.0000, 0.0000) &  1.5  && (\textbf{0.0000}, 0.0131, 0.0653) &  1.5\\ 
     rou20 &25.5 &&0.2021  & 0.9 & 362 && (\textbf{0.0000}, 0.0233, 0.1778) &  0.4  && (\textbf{0.0000}, 0.0074, 0.0193) &  0.4\\ 
     sko42 &26.9 &&0.2150  & 2.6 & 352 && (\textbf{0.0000}, 0.0076, 0.0253) &  4.2  && (\textbf{0.0000}, 0.0051, 0.0253) &  4.2\\ 
     sko49 &24.2 &&0.1197  & 2.9 & 284 && (\textbf{0.0000}, 0.0735, 0.1197) &  6.7  &&(0.0086, 0.0804, 0.1368) &  6.7\\ 
     sko56 &23.8 &&0.0348  & 5.2 & 327 && (0.0348, 0.0348, 0.0348) & 10.0  && (\textbf{0.0116}, 0.0818, 0.2264) & 10.0\\ 
     sko64 &21.3 &&0.0206  & 5.7 & 321 && (0.0124, 0.0198, 0.0206) & 15.0  && (\textbf{0.0041}, 0.0499, 0.1361) & 15.0\\ 
     sko72 &20.3 &&0.0664  & 7.7 & 380 && (0.0211, 0.0580, 0.0664) & 21.3  && (\textbf{0.0060}, 0.1210, 0.3139) & 21.4\\ 
    ste36a &138.8 &&2.2045  & 1.0 & 141 && (\textbf{0.0000}, 0.0042, 0.0420) &  2.6  && (\textbf{0.0000}, 0.0945, 0.2939) &  2.6\\ 
    ste36c &127.7 &&2.3769  & 1.2 & 183 && (\textbf{0.0000}, 0.0011, 0.0112) &  2.6  && (\textbf{0.0000}, 0.0188, 0.1883) &  2.6\\ 
    tai20a &27.5 &&0.9763  & 1.1 & 488 && (\textbf{0.0000}, 0.1852, 0.4697) &  0.4  && (\textbf{0.0000}, 0.1825, 0.3042) &  0.4\\ 
    tai25a &23.9 &&2.1781  & 1.5 & 498 && (\textbf{0.0000}, 0.0884, 0.3658) &  0.9  && (\textbf{0.0000}, 0.3898, 0.8721) &  0.9\\ 
    tai30a &21.0 &&1.7887  & 2.2 & 565 && (\textbf{0.0000}, 0.3436, 0.7806) &  1.5  && (\textbf{0.0000}, 0.3254, 0.6595) &  1.5\\ 
    tai30b &106.7 &&2.5101  & 0.9 & 183 && (\textbf{0.0000}, 0.0406, 0.1456) &  1.5  && (\textbf{0.0000}, 0.0782, 0.2637) &  1.5\\ 
    tai35a &21.1 &&2.3491  & 3.9 & 801 && (0.1642, 0.5192, 0.8617) &  2.4  && (\textbf{0.0671}, 0.5541, 0.9723) &  2.4\\ 
    tai35b &82.1 &&4.7863  & 1.3 & 217 && (\textbf{0.0000}, 0.0569, 0.2153) &  2.4  && (\textbf{0.0000}, 0.1596, 0.9763) &  2.4\\ 
    tai40a &20.6 &&1.6882  & 3.1 & 494 && (\textbf{0.3336}, 0.6905, 0.9928) &  3.6  &&(0.3753, 0.7424, 1.0622) &  3.6\\ 
    tai40b &77.8 &&0.0051  & 1.6 & 202 && (\textbf{0.0000}, 0.0000, 0.0000) &  3.6  && (\textbf{0.0000}, 0.2639, 2.1123) &  3.6\\ 
    tai50a &19.5 &&1.1252  & 6.2 & 715 && (\textbf{0.9638}, 1.0725, 1.1252) &  7.1  &&(0.9764, 1.2008, 1.4282) &  7.1\\ 
    tai50b &72.0 &&0.5660  & 5.1 & 591 && (0.0106, 0.3520, 0.5660) &  7.1  && (\textbf{0.0029}, 0.1101, 0.4161) &  7.1\\ 
    tai60a &18.2 &&1.0398  & 8.5 & 476 && (\textbf{0.8591}, 0.9631, 1.0398) & 12.3  &&(0.9338, 1.1456, 1.3461) & 12.3\\ 
    tai60b &66.0 &&0.0986  & 5.6 & 478 && (\textbf{0.0000}, 0.0732, 0.0986) & 12.3  &&(0.0010, 0.2838, 0.9002) & 12.3\\ 
     tho40 &42.0 &&0.2545  & 2.2 & 292 && (\textbf{0.0000}, 0.0523, 0.1172) &  3.6  && (\textbf{0.0000}, 0.0369, 0.0632) &  3.6\\ 
     wil50 &13.7 &&0.1311  & 6.9 & 799 && (\textbf{0.0000}, 0.0238, 0.0533) &  7.1  && (\textbf{0.0000}, 0.0315, 0.1024) &  7.1\\ 
\revise{$N_{\best}$}&&&\multicolumn{1}{l}{\quad \revise{ 0}} & & &&\multicolumn{1}{l}{\quad \revise{35}} & && \multicolumn{1}{l}{\quad  \revise{34}} & \\ 
\hline
     \end{longtable}
\end{footnotesize}
\end{center}

In Table \ref{table:lp:tabu:greater80:combination}, we present numerical results
on 21 problem instances with $n \geq 80.$ Table \ref{table:lp:tabu:greater80:combination} shows that our proposed $L_{0.75}$ regularization algorithm significantly outperforms Ro-TS in terms of both the solution quality and the speed for solving these problems. More specifically, our algorithm can find a permutation matrix whose gap is less than $0.8\%$ for all 21 problem instances and whose gap is less than $0.1\%$ for 11 problem instances. In contrast, the corresponding two numbers for Ro-TS (min) are 18 and 9 and for Ro-TS (mean) are only 15 and 2. For the largest problem instance ``tai256c'', our proposed algorithm is able to find a solution with gap $0.2610\%$ while the best gap returned by Ro-TS (among 10 runs) is $0.3169\%$, and our algorithm is 15 times faster than Ro-TS.

 \begin{table}[!htbp]
 \centering
 \linespread{1.0}
    \caption{Numerical results of $L_{0.75}$, $L_{0.75}$-Ro-TS, Ro-TS on 21 problem instances with $n \geq 80$ in QAPLIB} \label{table:lp:tabu:greater80:combination}
\begin{footnotesize}
  \begin{tabular}{@{}l @{\hspace{0.2mm}}r @{\hspace{0mm}}c @{\hspace{2.8mm}}   r @{\hspace{1.0mm}}r @{\hspace{1.0mm}}  r @{\hspace{0mm}} c @{\hspace{2.8mm}}  r @{\hspace{1.0mm}}  r @{\hspace{0mm}} c@{\hspace{2.8mm}} r@{\hspace{1.0mm}}r@{}}
\hline
 \Gape[4pt]Problem   &  &&   \multicolumn{3}{c}{$L_{0.75}$}  &&  \multicolumn{2}{c}{$L_{0.75}$-Ro-TS, 10 runs} &&  \multicolumn{2}{c}{Ro-TS, 10 runs} \\
\cline{1-2} \cline{4-6}    \cline{8-9} \cline{11-12}
\Gape[5pt] name &  Rgap  &&   gap&  time & nfe &&  gap&  $\overline{\text{time}}$ &  &   gap  (\%)& $\overline{\text{time}}$  \\
 &     &&  &     &  &&  (min, mean, max) &  &     &  (min, mean, max) &  \\
 \Xhline{0.8pt}
    esc128 &397.6 &&\textbf{0.0000}  &19.8 & 390 && (\textbf{0.0000}, 0.0000, 0.0000) &  0.0  && (\textbf{0.0000}, 13.4375, 28.1250) & 41.5\\ 
   lipa80a & 2.1 &&0.7540  & 0.4 & 37 && (\textbf{0.0000}, 0.4808, 0.5593) &  9.8  &&(0.5016, 0.5351, 0.5644) &  9.8\\ 
   lipa80b &30.8 &&\textbf{0.0000}  & 3.1 & 163 && (\textbf{0.0000}, 0.0000, 0.0000) &  0.0  && (\textbf{0.0000}, 0.0000, 0.0000) &  9.8\\ 
   lipa90a & 1.9 &&0.7179  & 0.9 & 37 && (0.4689, 0.4872, 0.5025) & 14.0  && (\textbf{0.4453}, 0.4734, 0.4980) & 14.1\\ 
   lipa90b &30.5 &&\textbf{0.0000}  & 2.2 & 84 && (\textbf{0.0000}, 0.0000, 0.0000) &  0.0  && (\textbf{0.0000}, 0.0000, 0.0000) & 14.1\\ 
   sko100a &17.4 &&\textbf{0.0645}  &14.5 & 396 && (0.0645, 0.0645, 0.0645) & 19.4  &&(0.0961, 0.1962, 0.5987) & 19.4\\ 
   sko100b &17.4 &&0.0858  &13.9 & 371 && (\textbf{0.0195}, 0.0489, 0.0819) & 19.4  &&(0.0871, 0.2102, 0.3808) & 19.4\\ 
   sko100c &18.0 &&0.0203  &16.8 & 464 && (\textbf{0.0108}, 0.0185, 0.0203) & 19.3  &&(0.0555, 0.2252, 0.4355) & 19.4\\ 
   sko100d &17.5 &&0.1712  &13.9 & 381 && (0.0949, 0.1296, 0.1645) & 19.4  && (\textbf{0.0388}, 0.2434, 0.4279) & 19.5\\ 
   sko100e &18.2 &&\textbf{0.0134}  &14.0 & 298 && (0.0134, 0.0134, 0.0134) & 19.3  &&(0.1824, 0.4008, 0.6195) & 19.4\\ 
   sko100f &17.3 &&0.0550  &12.3 & 328 && (\textbf{0.0550}, 0.0550, 0.0550) & 19.5  &&(0.1342, 0.3073, 0.5569) & 19.4\\ 
     sko81 &19.2 &&0.1143  &12.2 & 535 && (\textbf{0.0440}, 0.0963, 0.1143) & 10.2  &&(0.1165, 0.2024, 0.3143) & 10.2\\ 
     sko90 &18.5 &&0.0554  &14.7 & 526 && (0.0554, 0.0554, 0.0554) & 14.1  && (\textbf{0.0450}, 0.2612, 0.3653) & 14.1\\ 
   tai100a &14.4 &&\textbf{0.4131}  &33.8 & 968 && (0.4131, 0.4131, 0.4131) & 19.4  &&(0.9459, 1.1252, 1.3004) & 19.4\\ 
   tai100b &50.9 &&0.3800  &13.7 & 400 && (\textbf{0.3701}, 0.3786, 0.3800) & 19.4  &&(0.4035, 1.1940, 2.8654) & 19.4\\ 
   tai150b &30.6 &&0.5098  &86.2 & 1220 && (\textbf{0.4858}, 0.4899, 0.4926) & 67.1  &&(1.8733, 3.4037, 4.4771) & 66.9\\ 
   tai256c &19.8 &&\textbf{0.2610}  &24.6 & 191 && (0.2610, 0.2610, 0.2610) & 390.3  &&(0.3169, 0.4097, 0.5283) & 387.2\\ 
    tai80a &15.6 &&0.6904  &20.0 & 894 && (\textbf{0.5827}, 0.6571, 0.6904) &  9.8  &&(1.0769, 1.2615, 1.4734) &  9.8\\ 
    tai80b &51.9 &&0.0378  & 6.7 & 329 && (\textbf{0.0283}, 0.0283, 0.0283) &  9.8  &&(0.0379, 0.9246, 2.0504) &  9.8\\ 
    tho150 &20.6 &&\textbf{0.1500}  &53.4 & 679 && (0.1500, 0.1500, 0.1500) & 67.0  &&(0.2551, 0.6223, 1.0213) & 67.1\\ 
    wil100 & 9.8 &&0.0146  &21.0 & 542 && (\textbf{0.0110}, 0.0134, 0.0139) & 19.5  &&(0.1267, 0.1762, 0.2740) & 19.4\\ 
\revise{$N_{\best}$}&&&\multicolumn{1}{l}{\quad \revise{ 8}} & & &&\multicolumn{1}{l}{\quad \revise{13}} & && \multicolumn{1}{l}{\quad \quad  \revise{ 6}} & \\ 
\hline
     \end{tabular}
\end{footnotesize}
\end{table}

\section{Application in the bandwidth minimization problem} \label{section:bm}
In this section, we apply Algorithm \ref{alg:lp:practical} to solve the
bandwidth minimization (BM) problem. It is NP-hard \cite{papadimitriou1976np} and has many applications in different fields such as sparse matrix computations, circuit design, and VLSI layout \cite{chinn1982bandwidth, lai1999survey}.  For a real symmetric matrix $A \in \Rbb^{n \times n}$, its bandwidth is defined as $\mbox{b}(A)  = \max_{A_{ij} \neq 0} |i-j|$.  The BM problem is to find a permutation matrix $X$ such that the matrix $X^{\Tsf} AX$ takes the minimum bandwidth. Define $\bar A_{ij} = 1$ if $A_{ij} \neq 0$ and $\bar A_{ij} = 0$ otherwise. Clearly, $X^{\Tsf} \bar AX$ has the same bandwidth with $X^{\Tsf} A X$. Let $B_m \in \Rbb^{n \times n}$ be the symmetric Toeplitz matrix with  $(B_m)_{ij} = \max\{ |i - j| - m, 0\}$ and
define
 \be \label{equ:prob:bw}
   h(m) = \min\limits_{X \in \Pi_n} \, \tr(X^{\Tsf} \bar AX B_m).
 \ee
Then, $h(m) = 0$ if and only if the bandwidth of $X^{\Tsf} \bar AX$ is at most $m$. Therefore, the above BM problem for a given matrix $A$ can be formulated as the problem of finding the smallest nonnegative integer root of the equation $h(m) = 0$, namely,
\be\label{equ:prob:bw:root}
 \min \ m, \quad \st \quad h(m) = 0.
\ee

The BM problem can also be defined for a given graph. More specifically, for a graph ${\cal G}=({\cal V},{\cal E})$ with ${\cal V} = \{v_1, \ldots, v_n\}$ being  the vertices and $\cal E$ being the edges,  the BM problem is to find a bijection $\phi: {\cal V} \rightarrow \Ncal$
such that the number
$\max\limits_{(i,j) \in E}\, | \phi(v_i) -  \phi(v_j)|$ is minimal.  Denote the adjacency matrix of $\cal G$ by $A_{\cal G}$. Then the BM problem for the graph $\cal G$ is equivalent to the BM  problem for the matrix $A_{\cal G}$.
For more discussions on the BM problem for graphs and matrices, please refer to \cite{chinn1982bandwidth, van2014bounding}.

Let $m^*$ be the unique solution of \eqref{equ:prob:bw:root}. By the definition
of $h(m)$, we have that  $h(i) > h(j)$ for $0 \leq i < j \leq m^*$  and $h(i) =
0$ for $i \geq m^*$.  Moreover, it is trivial that $h(n-1) = 0$. Based on the
above properties of $h(m)$, we invoke the bisection procedure to solve
\eqref{equ:prob:bw:root}, where the subproblem \eqref{equ:prob:bw} is solved by
Algorithm \ref{alg:lp:practical}. The resulting method is named as Bi-$L_p$
 and described in Algorithm \ref{alg:lp:BiLp}.

    \begin{algorithm}[H]
  Set $\underline{m} = 0, \overline{m}=  n-1$ and the integer $m_0 \in (1, n-1)$, $k = 0$.\\
 \While{$\overline{m} - \underline{m} > 1$}{
 Compute $h(m_k)$ by Algorithm \ref{alg:lp:practical}.\\
 \lIf{$h(m_k)>0$}{$\um = m_k$}
 \lIf{$h(m_k)=0$}{$\om = m_k$}
 Set  $m_k= \frac12 \lceil \um + \om \rceil$  and $k = k + 1$.
}
   \caption{An Bi-$L_p$ algorithm for BM problem \eqref{equ:prob:bw:root}.}\label{alg:lp:BiLp}
     \end{algorithm}

Theoretically, if  problem \eqref{equ:prob:bw} is solved exactly,  the above bisection procedure will
also solve the BM problem   globally.  Even if problem
\eqref{equ:prob:bw} is only approximately solved (i.e., the $h(m_k)$ obtained is only an upper bound),  Algorithm \ref{alg:lp:BiLp} would still yield an upper bound on the minimum bandwidth.  
We can also invoke Algorithms
\ref{alg:lp:cp} -- \ref{alg:lp:cp:negProx} or other solvers for QAP to compute
$h(m_k)$. For simplicity,  we only use Algorithm \ref{alg:lp:practical} with $p = 0.75$ in Algorithm \ref{alg:lp:BiLp}.

We compare  the performance of Algorithm \ref{alg:lp:BiLp} with that of three
other solvers, including the MATLAB built-in function ``symrcm'' which
implements the reverse Cuthill-McKee (rCM) algorithm \cite{george1981computer},
the code ``imprRevCMcK'' (irCM) which is an improved rCM algorithm proposed in
\cite{van2014bounding}, and ``Bi-RoTS'' which utilizes the same  algorithm
framework as Algorithm  \ref{alg:lp:BiLp} but uses Ro-TS to compute $h(m_k)$.
The initial bandwidth guess $m_0$ in Bi-$L_{0.75}$ and Bi-RoTS is set to be the
minimal bandwidth returned by symrcm. We replace the implementation of
rCM in irCM by symrcm, which can improve the  performance of rCM based on our
experience.  The method irCM consists of  a number of  independent runs, which
is set to be 2000 in our tests. Each run of irCM first
generates a random ordering of the vertices of the graph, then performs the rCM algorithm and the improvement procedure developed in \cite{van2014bounding}.  In order to investigate the impact of this
number,  we run irCM until it takes twice as much CPU time as Bi-$L_{0.75}$ on
 each test instance. We denote the resulting method as irCM+. For some other methods, one can refer to \cite{mladenovic2010variable} and references therein.

The numerical results for three different kinds of bandwidth minimization
problems are reported in Tables \ref{table:lp:bandwidth:graph} --
\ref{table:lp:bandwidth:sparse}.  In these tables,  ``$\text{bw}^{\text{r}}$''
denotes the upper bound of the minimal bandwidth returned by symrcm and  ``bw''
denotes the upper bound of the minimal bandwidth returned by the other methods.
The running time  is counted in seconds. Since  symrcm is  fast, we do not
report its running time in these tables.  The term ``nrun'' denotes the total
number of independent runs of irCM+.   For each algorithm, the term
``$N_{\best}$'' denotes  the total number of problem instances when the upper
bound of the minimal bandwidth provided by this algorithm is the smallest among the bounds  provided by all the algorithms.  In Table \ref{table:lp:bandwidth:random}, ``$k$'' denotes the minimal bandwidth, and  we use ``$\overline{\phantom{ab}}$'' to denote the average values of each term over 10 instances.

\begin{table}[!htbp]
 \centering
 \linespread{1.0}
   \caption{Numerical results on the upper bound for the bandwidth of a few symmetric graphs
   } \label{table:lp:bandwidth:graph}
\begin{footnotesize}
 \begin{tabular}{
   @{}
    l@{
\hspace{1.5mm}}    r@{\hspace{3.0mm}}    r@{\hspace{0mm}}
    c@{\hspace{4.0mm}}  r@{\hspace{3.0mm}}  r@{\hspace{0mm}}
    c@{\hspace{4.0mm}}  r@{\hspace{3.0mm}}  r@{\hspace{0mm}}
    c@{\hspace{4.0mm}}  r@{\hspace{3.0mm}}  r@{\hspace{0mm}}
    c@{\hspace{4.0mm}}  r@{\hspace{3.0mm}}  r@{}}
\hline
 Problem   &  &   &&  \multicolumn{2}{c}{Bi-$L_{0.75}$} &&  \multicolumn{2}{c}{irCM}  &&  \multicolumn{2}{c}{irCM+} &&  \multicolumn{2}{c}{Bi-RoTS} \\
\cline{1-3} \cline{5-6}    \cline{8-9} \cline{11-12}  \cline{13-15}
\Gape[4pt] name &  $n$    &  $\text{bw}^{\text{r}}$ &&  bw &  time&&  bw &  time &  &   bw  & nrun &  &   bw  & time  \\
 \Xhline{0.8pt}
$H(3, 5)$& 125&  76&& \textbf{ 60} & 49.4&& \textbf{ 60} &  6.0&& \textbf{ 60} &  32795&&  64 & 197.7\\ 
$H(3, 6)$& 216& 145&& \textbf{101} & 281.9&& \textbf{101} & 72.4&& \textbf{101} &  15533&& 113 & 860.3\\ 
$H(4, 3)$&  81& \textbf{ 35}&& \textbf{ 35} & 13.6&& \textbf{ 35} &  0.3&& \textbf{ 35} & 209742&& \textbf{ 35} & 70.3\\ 
$H(4, 4)$& 256& 114&& \textbf{113} & 329.0&& \textbf{113} &  2.0&& \textbf{113} & 655738&& 114 & 1569.8\\ 
$GH(3, 4, 5)$&  60&  33&&  30 & 17.0&& \textbf{ 29} &  0.4&& \textbf{ 29} & 176368&& \textbf{ 29} & 41.8\\ 
$GH(4, 5, 6)$& 120&  68&&  59 & 77.2&& \textbf{ 57} &  4.1&& \textbf{ 57} &  72972&&  60 & 203.8\\ 
$GH(5, 6, 7)$& 210& 135&& \textbf{ 99} & 251.9&& \textbf{ 99} & 58.1&& \textbf{ 99} &  17100&& 114 & 783.6\\ 
$J(9, 3)$&  84& \textbf{ 49}&& \textbf{ 49} & 36.2&& \textbf{ 49} &  0.9&& \textbf{ 49} & 164372&& \textbf{ 49} & 77.5\\ 
$J(10, 3)$& 120&  71&& \textbf{ 68} & 98.1&& \textbf{ 68} &  3.2&& \textbf{ 68} & 124717&& \textbf{ 68} & 239.9\\ 
$J(11, 3)$& 165&  97&& \textbf{ 92} & 241.1&& \textbf{ 92} & 10.6&& \textbf{ 92} &  90684&&  97 & 638.8\\ 
$J(12, 3)$& 220& 127&& \textbf{120} & 601.1&& \textbf{120} & 36.2&& \textbf{120} &  65944&& 127 & 457.1\\ 
$J(8, 4)$&  70&  41&& \textbf{ 40} & 26.1&& \textbf{ 40} &  0.5&& \textbf{ 40} & 215414&& \textbf{ 40} & 107.4\\ 
$J(9, 4)$& 126& \textbf{ 70}&& \textbf{ 70} & 100.1&& \textbf{ 70} &  0.9&& \textbf{ 70} & 462981&& \textbf{ 70} & 161.3\\ 
$J(10, 4)$& 210& \textbf{110}&& \textbf{110} & 297.6&& \textbf{110} &  3.7&& \textbf{110} & 325865&& \textbf{110} & 588.8\\ 
$J(11, 4)$& 330& \textbf{170}&& \textbf{170} & 865.2&& 171 & 15.3&& \textbf{170} & 228949&& \textbf{170} & 3603.8\\ 
$K(9, 3)$&  84&  66&& \textbf{ 56} & 52.2&&  59 &  3.6&&  58 &  58243&& \textbf{ 56} & 62.2\\ 
$K(10, 3)$& 120&  96&& \textbf{ 86} & 90.4&&  90 &  7.5&&  90 &  49212&&  88 & 204.5\\ 
$K(11, 3)$& 165& 135&& \textbf{125} & 164.3&& 131 & 15.0&& 131 &  43851&& 129 & 548.4\\ 
$K(12, 3)$& 220& 183&& \textbf{173} & 315.6&& 181 & 27.9&& 180 &  44361&& 183 & 228.1\\ 
$K(9, 4)$& 126&  61&& \textbf{ 51} & 40.4&&  52 &  5.2&&  52 &  31173&&  52 & 165.1\\ 
$K(10, 4)$& 210& 169&& \textbf{106} & 377.9&& 118 & 72.1&& 117 &  21061&& 132 & 588.3\\ 
$K(11, 4)$& 330& 294&& \textbf{197} & 1289.7&& 223 & 229.5&& 219 &  21995&& 221 & 1787.7\\ 
$N_{\best}$ &&   5 && 20 & && 14 & && 15 &  &&  9 & \\ 
\hline
     \end{tabular}
\end{footnotesize}
\end{table}

\begin{table}[!htbp]
 \centering
 \linespread{1.0}
   \caption{Numerical results on the upper bound for the bandwidth of random symmetric sparse matrices
   } \label{table:lp:bandwidth:random}
\begin{footnotesize}
  \begin{tabular}{
   @{}
    l@{\hspace{-1.5mm}}    r@{\hspace{3.0mm}}    r@{\hspace{0mm}}
    c@{\hspace{4.0mm}}  r@{\hspace{3.0mm}}  r@{\hspace{0mm}}
    c@{\hspace{4.0mm}}  r@{\hspace{3.0mm}}  r@{\hspace{0mm}}
    c@{\hspace{4.0mm}}  r@{\hspace{3.0mm}}  r@{\hspace{0mm}}
    c@{\hspace{4.0mm}}  r@{\hspace{3.0mm}}  r@{}}
\hline
 Problem   &  &   &&  \multicolumn{2}{c}{Bi-$L_{0.75}$} &&  \multicolumn{2}{c}{irCM}  &&  \multicolumn{2}{c}{irCM+} &&  \multicolumn{2}{c}{Bi-RoTS} \\
\cline{1-3} \cline{5-6}    \cline{8-9} \cline{11-12}  \cline{13-15}

\Gape[9pt] \ $n$ &  $k$    &  $\overline{\text{bw}^{\text{r}}}$ &&  $\overline{\text{bw}}$ &  $\overline{\text{time}}$&&  $\overline{\text{bw}}$ &  $\overline{\text{time}}$ &  &   $\overline{\text{bw}}$  & $\overline{\text{nrun}}$&  &   $\overline{\text{bw}}$ & $\overline{\text{time}}$  \vspace{-2mm}\\
 \Xhline{0.8pt}
 80 &  15& 20.8&&  \textbf{15.4} &  5.7&& 16.4 &  1.1&& 16.4 & 21129&& 15.4 & 28.9\\ 
 80 &  23& 31.6&&  \textbf{23.0} &  7.7&& 24.6 &  1.8&& 24.6 & 18648&& 23.5 & 26.9\\ 
 80 &  31& 39.9&&  \textbf{31.4} & 10.1&& 33.5 &  2.2&& 33.5 & 19925&& 31.4 & 36.0\\ 
 80 &  39& 51.0&&  \textbf{39.0} & 13.4&& 41.7 &  2.3&& 41.7 & 24668&& 39.4 & 32.4\\ 
100 &  19& 26.4&&  \textbf{19.4} & 10.5&& 20.2 &  1.7&& 20.1 & 26017&& 19.5 & 63.7\\ 
100 &  29& 38.7&&  \textbf{29.4} & 13.2&& 31.1 &  2.9&& 31.1 & 21398&& 29.6 & 65.1\\ 
100 &  39& 51.3&&  \textbf{39.4} & 16.3&& 41.9 &  3.8&& 41.9 & 18292&& 39.9 & 73.5\\ 
100 &  49& 61.6&&  \textbf{49.4} & 24.6&& 53.0 &  4.7&& 53.0 & 23037&& 49.5 & 71.1\\ 
200 &  39& 52.5&&  \textbf{39.5} & 73.4&& 41.4 & 13.8&& 41.4 & 23358&& 40.3 & 555.2\\ 
200 &  59& 75.9&&  \textbf{59.5} & 73.2&& 62.9 & 23.7&& 62.9 & 12882&& 60.8 & 718.7\\ 
200 &  79& 99.3&&  \textbf{79.4} & 91.7&& 84.1 & 35.6&& 84.1 & 11786&& 80.8 & 657.8\\ 
200 &  99& 130.3&&  \textbf{99.4} & 134.6&& 105.6 & 51.6&& 105.6 & 11191&& 101.3 & 635.5\\ 
300 &  59& 75.1&&  \textbf{59.5} & 235.0&& 63.3 & 37.3&& 63.3 & 28214&& 60.9 & 2733.1\\ 
300 &  89& 112.4&&  \textbf{89.6} & 279.2&& 94.0 & 131.3&& 94.0 & 8948&& 90.9 & 2679.4\\ 
300 & 119& 154.9&&  \textbf{119.6} & 334.0&& 125.9 & 203.5&& 125.9 & 7050&& 120.9 & 3142.0\\ 
300 & 149& 190.4&&  \textbf{149.5} & 592.5&& 157.1 & 280.9&& 157.1 & 8847&& 151.4 & 2748.5\\ 
$N_{\best}$ &&   0 && 16 & && 0 & && 0 &  &&  0 &  \\ \hline     \end{tabular}
\end{footnotesize}
\end{table}

Table \ref{table:lp:bandwidth:graph} reports the results on several symmetric graphs, including  the Hamming graph, the 3-dimensional generalized Hamming graph, the Johnson graph and the Kneser graph. For details on these graphs, one can refer to  \cite{van2014bounding}.
Table  \ref{table:lp:bandwidth:random} reports the results on some random
symmetric sparse matrices.  For each $n$ and the minimal bandwidth $k$,  we
generate 10 random matrices by the MATLAB commands
       ``A = toeplitz([ones(k,1); zeros(n-k,1)]);
        T = rand(n) $>$ 0.4;  T = (T + T');
        AS = A.*T;  ps = randperm(n);
        A = AS(ps,ps); ''
 In Table \ref{table:lp:bandwidth:sparse}, we present the results on 49 sparse matrices from the UF Sparse Matrix Collection \cite{davis2011university}.  From these three tables, we can observe that Bi-$L_{0.75}$ performs the best in terms of the solution quality while Bi-RoTS performs the worst.    In particular, for the  Kneser graphs and the random sparse matrices,   Bi-$L_{0.75}$  returns much better bounds than those of irCM and irCM+.

\begin{center}
 \linespread{1.0}
 \begin{footnotesize}
  \begin{longtable}[!htbp]{
  @{}
    l@{\hspace{1.5mm}}    r@{\hspace{3.0mm}}    r@{\hspace{0mm}}
    c@{\hspace{4.0mm}}  r@{\hspace{3.0mm}}  r@{\hspace{0mm}}
    c@{\hspace{4.0mm}}  r@{\hspace{3.0mm}}  r@{\hspace{0mm}}
    c@{\hspace{4.0mm}}  r@{\hspace{3.0mm}}  r@{\hspace{0mm}}
    c@{\hspace{4.0mm}}  r@{\hspace{3.0mm}}  r@{}}
 \caption{Numerical results on the upper bound for the bandwidth of 49 symmetric sparse matrices with $80 \leq n \leq 300$ from the UF Sparse Matrix Collection \cite{davis2011university}} \label{table:lp:bandwidth:sparse}
    ~\\[0.6pt]
 \hline
 Problem   &  &   &&  \multicolumn{2}{c}{Bi-$L_{0.75}$} &&  \multicolumn{2}{c}{irCM}  &&  \multicolumn{2}{c}{irCM+} &&  \multicolumn{2}{c}{Bi-RoTS} \\
\cline{1-3} \cline{5-6}    \cline{8-9} \cline{11-12}  \cline{13-15}
\Gape[4pt] name &  $n$    &  $\text{bw}^{\text{r}}$ &&  bw &  time&&  bw &  time &  &   bw  & nrun &  &   bw  & time  \\
 \Xcline{1-15}{0.8pt}
 \endfirsthead
   \\[6pt]
\hline
 Problem   &  &   &&  \multicolumn{2}{c}{Bi-$L_{0.75}$} &&  \multicolumn{2}{c}{irCM}  &&  \multicolumn{2}{c}{irCM+} &&  \multicolumn{2}{c}{Bi-RoTS} \\
\cline{1-3} \cline{5-6}    \cline{8-9} \cline{11-12}  \cline{13-15}
\Gape[4pt] name &  $n$    &  $\text{bw}^{\text{r}}$ &&  bw &  time&&  bw &  time &  &   bw  & nrun &  &   bw  & time  \\
 \Xcline{1-15}{0.8pt}
 \endhead
 \multicolumn{15}{r}{Continued on next page}
 \\ \hline
\endfoot
\endlastfoot
GD06\_theory & 101& 57&&  \textbf{34} & 39.8&& 35 & 15.5&& 35 & 10266&&  \textbf{34} & 65.1\\ 
   GD97\_a & 84& 23&& 18 &  9.8&& 18 &  1.1&&  \textbf{17} & 36174&&  \textbf{17} & 47.6\\ 
   GD98\_c & 112& 50&&  \textbf{26} & 34.1&& 32 & 10.9&& 32 & 12569&& 28 & 111.0\\ 
  Journals & 124& 116&&  \textbf{98} & 77.1&& 102 &  4.2&& 102 & 73658&& 103 & 151.5\\ 
Sandi\_authors & 86& 22&&  \textbf{12} & 17.3&&  \textbf{12} &  9.6&&  \textbf{12} & 7177&& 13 & 36.0\\ 
Trefethen\_150 & 150& 79&&  \textbf{67} & 91.4&& 71 &  1.3&& 70 & 289009&& 79 & 203.2\\ 
Trefethen\_200 & 200& 99&&  \textbf{87} & 167.7&& 88 &  1.5&& 88 & 461919&& 99 & 678.5\\ 
Trefethen\_200b & 199& 99&&  \textbf{86} & 243.1&& 87 &  1.5&& 87 & 643423&& 99 & 667.1\\ 
Trefethen\_300 & 300& 152&&  \textbf{132} & 652.5&&  \textbf{132} &  3.2&&  \textbf{132} & 817720&& 150 & 3973.7\\ 
   adjnoun & 112& 67&&  \textbf{38} & 32.5&& 42 & 20.8&& 42 & 6225&& 42 & 55.9\\ 
    ash292 & 292& 32&& 23 & 316.3&&  \textbf{21} & 26.6&&  \textbf{21} & 47741&& 32 & 3718.3\\ 
     ash85 & 85& 13&& 10 & 10.0&& 10 &  0.6&& 10 & 70416&&  \textbf{ 9} & 47.3\\ 
  bcspwr03 & 118& 17&& 12 & 19.6&&  \textbf{11} & 11.8&&  \textbf{11} & 6459&& 17 & 193.5\\ 
  bcspwr04 & 274& 49&& 31 & 304.1&&  \textbf{29} & 207.1&&  \textbf{29} & 5876&& 49 & 2509.4\\ 
  bcsstk03 & 112& \textbf{ 3}&&  \textbf{ 3} &  3.8&&  \textbf{ 3} &  0.3&&  \textbf{ 3} & 54309&&  \textbf{ 3} & 162.2\\ 
  bcsstk04 & 132& 54&&  \textbf{37} & 52.4&& 40 &  3.5&& 40 & 58609&& 38 & 184.2\\ 
  bcsstk05 & 153& 24&&  \textbf{20} & 41.8&&  \textbf{20} &  4.7&&  \textbf{20} & 35781&& 24 & 294.0\\ 
  bcsstk22 & 138& 12&& 11 & 63.5&&  \textbf{10} &  2.4&&  \textbf{10} & 102888&& 12 & 158.5\\ 
  can\_144 & 144& 18&&  \textbf{13} & 25.9&&  \textbf{13} &  8.5&&  \textbf{13} & 12159&& 18 & 247.4\\ 
  can\_161 & 161& \textbf{18}&&  \textbf{18} & 30.1&& 19 &  1.3&&  \textbf{18} & 90297&&  \textbf{18} & 258.9\\ 
  can\_187 & 187& 23&&  \textbf{14} & 62.0&& 15 &  2.0&& 15 & 123730&& 23 & 546.6\\ 
  can\_229 & 229& 37&&  \textbf{32} & 155.4&&  \textbf{32} & 16.1&&  \textbf{32} & 39029&& 37 & 1041.5\\ 
  can\_256 & 256& 123&& 67 & 405.7&&  \textbf{62} & 249.2&&  \textbf{62} & 6561&& 85 & 1179.5\\ 
  can\_268 & 268& 98&& 57 & 359.7&&  \textbf{53} & 334.7&&  \textbf{53} & 4256&& 82 & 1836.2\\ 
  can\_292 & 292& 76&&  \textbf{40} & 716.8&& 50 & 50.2&& 50 & 56878&& 76 & 628.3\\ 
   can\_96 & 96& 23&&  \textbf{13} & 16.0&& 21 &  0.6&& 21 & 99598&&  \textbf{13} & 39.5\\ 
  dwt\_162 & 162& 16&&  \textbf{13} & 72.0&& 14 &  2.8&& 14 & 102152&& 16 & 174.3\\ 
  dwt\_193 & 193& 54&&  \textbf{32} & 145.3&& 34 & 56.0&& 34 & 10505&& 42 & 891.7\\ 
  dwt\_198 & 198& 13&&  \textbf{ 8} & 70.5&& 10 &  3.4&& 10 & 83262&& 13 & 488.2\\ 
  dwt\_209 & 209& 33&&  \textbf{26} & 97.9&& 29 & 12.8&& 29 & 30822&& 33 & 773.0\\ 
  dwt\_221 & 221& 15&&  \textbf{14} & 97.8&&  \textbf{14} &  2.8&&  \textbf{14} & 139580&& 15 & 1152.0\\ 
  dwt\_234 & 234& 24&&  \textbf{12} & 110.3&&  \textbf{12} & 13.5&&  \textbf{12} & 32277&& 24 & 990.3\\ 
  dwt\_245 & 245& 55&&  \textbf{28} & 267.1&& 31 & 132.3&& 31 & 8100&& 53 & 1609.1\\ 
   dwt\_87 & 87& 18&& 12 &  9.1&& 12 &  4.0&& 12 & 9072&&  \textbf{11} & 38.6\\ 
  football & 115& 76&&  \textbf{37} & 43.8&& 42 & 10.6&& 42 & 16866&& 39 & 157.3\\ 
     grid1 & 252& 22&& 22 & 409.5&&  \textbf{19} &  0.7&&  \textbf{19} & 2200545&& 22 & 1799.5\\ 
grid1\_dual & 224& \textbf{17}&&  \textbf{17} & 162.2&&  \textbf{17} &  0.7&&  \textbf{17} & 984859&&  \textbf{17} & 1195.5\\ 
      jazz & 198& 141&&  \textbf{69} & 162.2&& 81 & 105.1&& 81 & 6188&& 88 & 324.7\\ 
 lshp\_265 & 265& 18&& 18 & 129.2&&  \textbf{17} &  0.8&&  \textbf{17} & 612015&& 18 & 951.9\\ 
   lund\_a & 147& \textbf{23}&&  \textbf{23} & 42.4&&  \textbf{23} &  0.4&&  \textbf{23} & 396257&&  \textbf{23} & 263.5\\ 
   lund\_b & 147& \textbf{23}&&  \textbf{23} & 49.5&&  \textbf{23} &  0.4&&  \textbf{23} & 462959&&  \textbf{23} & 260.3\\ 
   mesh3e1 & 289& \textbf{17}&&  \textbf{17} & 192.7&&  \textbf{17} &  0.9&&  \textbf{17} & 839178&&  \textbf{17} & 2368.1\\ 
  mesh3em5 & 289& \textbf{17}&&  \textbf{17} & 189.4&&  \textbf{17} &  0.9&&  \textbf{17} & 824889&&  \textbf{17} & 2367.0\\ 
      nos1 & 237&  4&&  \textbf{ 3} & 43.8&&  \textbf{ 3} &  1.1&&  \textbf{ 3} & 156608&&  4 & 1706.8\\ 
      nos4 & 100& 12&&  \textbf{10} &  9.9&&  \textbf{10} &  4.2&&  \textbf{10} & 9422&& 11 & 97.1\\ 
  polbooks & 105& 30&&  \textbf{20} & 17.7&& 21 & 12.2&& 21 & 5883&& 22 & 112.4\\ 
spaceStation\_1 & 99& 55&&  \textbf{28} & 24.1&& 33 &  3.9&& 33 & 25099&& 30 & 90.5\\ 
   sphere3 & 258& 32&& 30 & 328.0&& 28 &  1.7&&  \textbf{27} & 744388&& 32 & 1649.3\\ 
tumorAntiAngiogenesis\_1 & 205& 199&&  \textbf{100} & 170.6&&  \textbf{100} & 67.0&&  \textbf{100} & 10212&& 101 & 542.7\\ 
$N_{\best}$ &&   7 && 37 & && 24 & && 27 &  &&  12 &  \\ \hline
     \end{longtable}
\end{footnotesize}
\end{center}

\section{Concluding remarks} \label{secion:conclusion}
In this paper, we considered optimization problem \eqref{equ:prob:optperm} over
permutation matrices. We proposed an $L_p$ regularization model for problem
\eqref{equ:prob:optperm}. We studied theoretical properties of the proposed
model, including its exactness, the connection between its local minimizers and
the permutation matrices, and the lower bound theory of nonzero elements of its
KKT points. Based on the $L_p$ regularization model, we proposed an efficient
$L_p$ regularization algorithm for solving problem \eqref{equ:prob:optperm}. To
further improve the performance of the proposed $L_p$ regularization algorithm,
we combined it with the classical cutting plane technique and/or a novel
negative proximal point technique. Our numerical results showed that the
proposed $L_p$  regularization algorithm and its variants performed quite well for QAP instances from QAPLIB.  For the bandwidth minimization problem which has many applications in different fields, our proposed algorithms also exhibit satisfactory performance.

Notice that the main computational cost of our proposed algorithms
arises  from many projections onto $\Dcaln$ (or the restricted domain $\Acal$). It would
be helpful to develop faster algorithms for computing these projection
especially when $n$ is large. Moreover, how to take full advantage of the low rank and/or sparse structure of
the matrices $A$ and $B$ in our algorithms remains further investigation.

Our $L_p$ regularization model/algorithm for solving problem
\eqref{equ:prob:optperm} can be directly extended to solve a general nonlinear
binary programming with a fixed number of nonzero elements as:
\be
 \min_{x \in \Omega}\, f(x)\quad \st \quad x_i \in \{0,1\},~i=1,2,\ldots,n,\ \|x\|_0 = N, \nn
\ee
where $\Omega \subset \Rbb^n$ is convex and $N$ is given.  The $L_p$ regularization model for the above problem is
\be
 \min_{x \in \Omega}\, f(x) + \sigma \|x+\epsilon\e\|_p^p \quad \st \quad  0\leq x_i\leq 1,~i=1,2,\ldots,n, \e^{\Tsf} x = N, \nn
\ee and it can be solved efficiently by the projected gradient method.
It would be interesting to extend  the $L_p$-norm regularization  to solve some
general nonlinear integer programming problems.

\section*{Acknowledgements}
We thank Yong Xia for the valuable discussions  on the quadratic
assignment problem and for sharing his code LagSA. We also thank Renata Sotirovr for the helpful discussions on
the bandwidth minimization problem and for sharing her code imprRevCMcK. The authors are grateful to the Associate Editor  Alexandre d'Aspremont  and two anonymous referees for their detailed and valuable comments and suggestions.

\bibliography{qaplp}
\bibliographystyle{siam}

\end{document}